\documentclass[11pt, oneside, reqno]{amsart}
\usepackage{geometry}                
\usepackage{graphicx}
\usepackage{amssymb, amsthm, mathrsfs, bm, overpic, upgreek}
\usepackage{epstopdf}
\usepackage{enumerate}
\usepackage{stmaryrd}
\usepackage[usenames,dvipsnames,svgnames,table]{xcolor}
\usepackage{hyperref}
\newcommand{\be}{\begin{equation*}}
\newcommand{\ee}{\end{equation*}}
\newcommand{\ben}[1]{\begin{equation}\label{#1}}
\newcommand{\een}{\end{equation}}
\newcommand{\bea}{\begin{eqnarray}}
\newcommand{\eea}{\end{eqnarray}}
\newcommand{\bean}{\begin{eqnarray*}}
\newcommand{\eean}{\end{eqnarray*}}

\newcommand{\R}{\mathbb{R}}

\newcommand{\C}{\mathbb{C}}
\newcommand{\p}{\partial}

\newcommand{\N}{\mathbb{N}}

\newcommand{\cc}{\subset \subset}
\newcommand{\abs}[1]{\left|#1 \right|} 
\newcommand{\norm}[2]{\left|\left |#1 \right| \right |_{#2}} 
\newcommand{\snorm}[2]{\left[ #1 \right ]_{#2}} 
\newcommand{\ssnorm}[2]{\left \llbracket #1 \right\rrbracket_{#2}}

\newcommand{\eq}[1]{(\ref{#1})}
\newtheorem{Theorem}{Theorem}
\newtheorem{Proposition}{Proposition}

\newtheorem{Lemma}{Lemma}
\newtheorem{Corollary}[Theorem]{Corollary}
\newtheorem{Definition}{Definition}
\numberwithin{Theorem}{section}
\numberwithin{Lemma}{section}

\title[Model Problem for QNMs]{A model problem for quasinormal ringdown of asymptotically flat or extremal black holes}
\author{Dejan Gajic${}^{{*}}$, Claude Warnick${}^\dagger$}
\thanks{$^{*}$\texttt{d.gajic@dpmms.cam.ac.uk; ${}^\dagger$c.m.warnick@maths.cam.ac.uk.} \\
\phantom{1 } Centre for Mathematical Sciences, Wilberforce Road, Cambridge, CB3 0WA, UK}

\begin{document}
\maketitle
\begin{abstract}
We consider a wave equation with a potential on the half-line as a model problem for wave propagation close to an extremal horizon, or the asymptotically flat end of a black hole spacetime. We propose a definition of quasinormal frequencies (QNFs) as eigenvalues of the generator of time translations for a null foliation, acting on an appropriate (Gevrey based) Hilbert space. We show that this QNF spectrum is discrete in a subset of $\C$ which includes the region $\{\Re{s} >-b, \abs{\Im{s}}> K\}$ for any $b>0$ and some $K=K(b) \gg 1$. As a corollary we establish the meromorphicity of the scattering resolvent in a sector $\abs{\arg{s}} <\varphi_0$ for some $\varphi_0 > \frac{2\pi}{3}$, and show that the poles occur only at quasinormal frequencies according to our definition. Finally, we show that QNFs computed by the continued fraction method of Leaver are necessarily QNFs according to our new definition. This paper is a companion to [D. Gajic, C. Warnick, \href{https://arxiv.org/abs/1910.08481}{arXiv:1910.08481}] which deals with the QNFs of the wave equation on the extremal Reissner N\"ordstrom black hole.
\end{abstract}

\section{Introduction}

Consider the following wave equation for $t \geq 0$ and $r$ taking values in the half-line $\R_{>1}:=(1, \infty)$:
\ben{we}
-\frac{1}{4} \frac{\p^2 \psi}{\p \tau^2} + \frac{\p^2 \psi}{\p r^2} - V(r) \psi = 0.
\een
The factor of $\frac{1}{4}$ is chosen for convenience and can be removed by a change of coordinates. We impose the Dirichlet boundary condition $\psi(1, \tau) = 0$. The potential $V\geq 0$ will be assumed to be smooth on $\R_{\geq1}:=[1, \infty)$, and to satisfy some decay conditions as $r \to \infty$, which we will specify shortly. We can think of this as a convenient model problem for the types of wave equation that arise in studying black hole perturbations, after restricting to fixed angular frequency. The end at `$r=\infty$' corresponds in this setting to a black hole horizon or alternatively to an asymptotic end, and the nature of this end is characterised by the asymptotic behaviour of $V$. Loosely, we shall consider two possibilities: type I potentials which have exponential fall-off and admit an asymptotic expansion in powers of $e^{-\kappa r}$ near infinity, and type II potentials which fall off like $r^{-2}$ near infinity and admit an asymptotic expansion in powers of $r^{-1}$. The barrier at $r=1$ is artificial, but permits us to restrict attention to one horizon/asymptotic end at a time.

We shall briefly describe the definition of quasinormal modes for type I potentials, which arise when `$r=\infty$' corresponds to a non-degenerate Killing horizon, such as a non-extremal black hole horizon or a de Sitter horizon. This definition follows from the work of \cite{Vasy2013, CW, Gannot}. It is the definition of quasinormal modes for type II potentials which occupies the bulk of the paper, and which is relevant for the situation of asymptotically flat or extremal black holes. This paper treats the model problem in detail, and is a companion to \cite{CD} which studies the extremal Reissner-Nordstrom black hole, and in particular includes results concerning the full three-dimensional problem (rather than the symmetry reduced one-dimensional problem). For a full review of the literature in the context of asymptotically flat and extremal black holes, we refer the reader to \cite{CD}.

To understand the late-time behaviour of solutions to \eq{we}, it is useful to consider the Laplace transformed operator:
\[
\hat{L}_sw := \frac{d^2w}{dr^2} - \left(\frac{s^2}{4} + V(r)\right)w.
\]
Standard theory gives that $\hat{L}_s : H^2\cap H^1_0(\R_{>1}) \to L^2(\R_{>1})$ is invertible, and moreover $\hat{L}_s^{-1}:L^2(\R_{>1}) \to H^2(\R_{>1})$ is a holomorphic family of operators on the half-plane $\{\Re(s)>0\}$.

Let us for a moment assume that $V$ has support in $r<R$. Then $\hat{L}_s^{-1}$ can be explicitly constructed by Green's function methods, and it is possible to show that the operator may be meromorphically continued to $\C$, as an operator with modified domain and range $\hat{L}_s^{-1}:L^2_c(\R_{>1}) \to H^2_{loc.}(\R_{>1})$. This extension is known as the scattering resolvent. The late time behaviour of solutions of \eq{we} is closely related to the singularity structure of this continuation. In particular, the location of the poles of the scattering resolvent (known as the scattering resonances) encode information about `ringdown' behaviour.

The scattering resonances occur precisely at those values of $s$ for which there exists a resonant state, that is a solution $w$ to $\hat{L}_s w =0$ which satisfies the Dirichlet condition at $r=1$ and is outgoing in the sense that $w(r) = A e^{-\frac{s}{2} r}$ for $r>R$, where $A\in \C$ is a constant.  The resonant states corresponding to each scattering resonance give rise to time-harmonic solutions, $\psi(r, \tau) = e^{s \tau} w(r)$, to $\eq{we}$, and the late time behaviour of a general solution can be approximated as a sum over such time harmonic solutions (see \cite{DyatlovZworski} \S2.3).

If we relax the assumption that $V$ has compact support, then defining the scattering resonances becomes more difficult. Establishing the necessary analyticity properties to perform the meromorphic continuation requires some work, and relatedly the `outgoing' boundary condition becomes rather subtle. Even in the case of compactly supported $V$, we note that for a scattering resonance, the corresponding resonant state $w$ grows exponentially as $r \to \infty$. These issues can be resolved through the method of complex scaling (known in numerical settings as the perfectly matched layer method) at least if the potential is assumed to be real analytic outside some compact region (see \cite{DyatlovZworski}, \S2.7 and references therein). We shall consider an alternative approach, which has the benefit of requiring a weaker assumption than analyticity. We mention, however, \cite{GalZwo} which appeared subsequent to this paper appearing as a preprint and which is discussed below, which addresses potentials which fail to be analytic in a similar fashion, and \emph{does} make use of complex scaling.

\subsection{Type I Potentials} We will shortly give a precise definition, but for now, we say that $V$ is type I if there exist $\kappa >0$ and constants $V_n \in \C$ such that 
\[
V(r) \sim \sum_{k=1}^\infty V_ne^{- \kappa k r}\qquad \textrm{ as }r \to \infty.
\]
Such asymptotic behaviour models the situation where `$r=\infty$' corresponds to a non-degenerate Killing horizon of surface gravity $\kappa$, or an asymptotic end of an even asymptotically hyperbolic geometry (see \cite[Chapter 5]{DyatlovZworski}). For type I potentials, the existence of a meromorphic continuation for the scattering resolvant $\hat{L}_s^{-1}:L^2_c(\R_{>1}) \to H^2_{loc.}(\R_{>1})$ can be established in various ways, but for our purposes the most relevant are the approach of \cite{Vasy2013, CW, Gannot}. One way to understand these approaches is to make use of a different set of coordinates for the original time dependent equation \eq{we}. For an alternative, more detailed, and considerably more general exposition see \cite[\S5.2]{DyatlovZworski}. In our model problem we make the change of coordinates
\[
t =\tau - \frac{r}{2} , \qquad x = e^{- \kappa r}.
\]
The equation becomes:
\ben{nullwe1}
\frac{\p}{\p x} \left( \kappa x \frac{\p \psi}{\p x}\right) + \frac{\p^2 \psi}{\p x \p t} - W \psi = 0,
\een
where $W(x) := \frac{1}{\kappa x} V\left(- \frac{1}{\kappa} \log x\right)$. This motivates:
\begin{Definition}
$V:\R_{\geq 1} \to \R$ is a type I potential if there exists $\kappa >0$ such that the function $W:(0, e^{-\kappa}] \to \R$ given by $W(x) := \frac{1}{\kappa x} V\left(- \frac{1}{\kappa} \log x\right)$ extends to a smooth function on $[0, e^{-\kappa}]$.
\end{Definition}

Defining a Laplace transformed operator with respect to the $t$ variable, we wish to consider:
\[
\tilde{L}_su := \frac{d}{d x} \left( \kappa x \frac{d u }{d x}\right) + s \frac{du }{d x} - Wu.
\]
The crucial insight of  \cite{Vasy2013, CW, Gannot} is that the operator $\tilde{L}_s$ is Fredholm in the half-plane $\{\Re{s} > -\kappa \left( k-\frac{1}{2}\right) \}$, when $\tilde{L}_s$ is understood as a closed unbounded operator on $H^{k-1}(J)$, where $J = (0, e^{-\kappa})$. More concretely, let $\tilde{\mathcal{D}}^k$ be the completion of  $\{ u \in C^\infty(\overline{J}) | u(e^{-\kappa})=0\}$ with respect to the graph norm
\[
\norm{u}{k} = \norm{u}{H^{k-1}(J)} + \norm{\tilde{L}_s u}{H^{k-1}(J)}.
\]
Then for $\Re{s} > -\kappa \left( k-\frac{1}{2}\right)$, we have that $\tilde{\mathcal{D}}^k \subset H^{k}(J)$ is independent of $s$ and $\tilde{L}_s: \tilde{\mathcal{D}}^k \to H^{k-1}(J)$ is an analytic family of Fredholm operators. This can be established (in the language of \cite{CW}) by making use of energy estimates for the time dependent problem, and in particular exploiting the redshift effect at the non-degenerate Killing horizon $x=0$. See, for example, \cite[\S1.4]{CW} which discusses the result for a closely related operator.

From this result, it follows that $\tilde{L}_s^{-1}:C^\infty(\overline{J}) \to C^\infty(\overline{J})$ is a meromorphic family of operators, whose poles correspond to functions $u \in C^\infty(\overline{J})$  which satisfy $u(e^{-\kappa}) = 0$ and $\tilde{L}_s u = 0$.  As a direct corollary, it can be shown that $\hat{L}_s^{-1}:L^2_c(\R_{>1}) \to H^2_{loc.,}(\R_{>1})$ extends meromorphically from $\{\Re(s)>0\}$ to $\C$, with each pole located at a pole of $\tilde{L}_s^{-1}:C^\infty(\overline{J}) \to C^\infty(\overline{J})$.  To each pole of $\hat{L}_s^{-1}$ is associated a resonant state, that is a solution to $\hat{L}_sw = 0$ which satisfies the Dirichlet condition at $r=1$ and is `outgoing' in the sense that $w(r) = e^{-\frac{s}{2} r} u(e^{- \kappa r})$ where $u \in C^\infty(\overline{J})$. By Taylor's theorem applied at $x=0$ we see that this implies the existence of $w_n \in \C$ such that 
\ben{og1}
w(r) \sim e^{-\frac{s}{2} r}\sum_{k=0}^\infty w_ne^{- \kappa k r}\qquad \textrm{ as }r \to \infty,
\een
We note that this condition effectively picks one of the two asymptotic branches of the general solution to $\hat{L}_s w = 0$ near $r = \infty$. In particular, the solution with leading order behaviour $e^{\frac{s}{2} r}$ near $r = \infty$ will not satisfy the `outgoing' condition (we ignore here the points $s \in -\kappa \N$, which give rise to some additional subtleties).

We have, furthermore, gained something in our interpretation of the scattering resonances. We can show that solving \eq{nullwe1} naturally gives rise to a $C^0-$semigroup $\mathcal{S}(t): H^{k}(J)\to H^{k}(J)$, whose generator is given by:
\[
\mathcal{A} u = \int_x^1\left[   \p_\xi(\kappa \xi \p_\xi u)- W(\xi) u(\xi)\right]d \xi
\]
so that $\mathcal{S}(t) = e^{\mathcal{A} t}$. The domain of $\mathcal{A}$ is precisely $\tilde{\mathcal{D}}^k$ and moreover $u\in \tilde{\mathcal{D}}^k$ is an eigenfunction of $\mathcal{A}$ with eigenvalue $s$ if and only if $\tilde{L}_s u = 0$.  We say $s$ belongs to the $H^k-$quasinormal spectrum, $s \in \Lambda_{QNF}^k$, if  $\Re{s} > -\kappa \left( k-\frac{1}{2}\right)$ and $s$ is an eigenvalue of $\mathcal{A}: \tilde{\mathcal{D}}^k \to H^{k}(J)$. We have $\Lambda_{QNF}^{k+1} \cap \{\Re{s} > -\kappa \left( k-\frac{1}{2}\right)\} = \Lambda_{QNF}^{k}$, and we define the quasinormal spectrum, or the set of quasinormal frequencies, to be $\Lambda_{QNF}= \cup_{k} \Lambda_{QNF}^{k}$.

The scattering resonances (defined as poles of a meromorphic continuation of $\hat{L}_s$) are a subset of the quasinormal frequencies (defined as eigenvalues of $\mathcal{A}$). Generically one expects the two sets to coincide, however there are situations where the scattering resonances are a strict subset of the quasinormal frequencies \cite[\S6]{CW} (in this example, the discrepancy arises at the set $-\kappa\N$).

\subsection{Type II potentials}

Having briefly set out the situation for Type I potentials, we move on to the novel results of this paper, which concern a class of potentials which decay polynomially in $r$. We shall give a broader definition of our class of potentials in the sequel, but for now it suffices to consider $V$ of the form:
\[
V(r) = \frac{1}{r^2} \sum_{k=0}^p \frac{V_k}{r^k} 
\]
for some $V_k \in \R$ and $p \in \N$. These potentials model the situation where `$r=\infty$' is either an extremal black hole (with vanishing surface gravity) or else an asymptotically flat end. For these potentials, we shall establish:
\begin{Theorem}\label{Thm1}
The resolvent $\hat{L}_s^{-1}:\left ( \frac{\p^2}{\p r^2} - V(r)-\frac{s^2}{4}\right)^{-1}:L^2(\R_{> 1})  \to  H^2(\R_{> 1})$, which is holomorphic for $\Re(s)>0$, admits a meromorphic extension as an operator from $L^2_c(\R_{> 1})$ to $H^2_{loc.}(\R_{> 1})$  for $s$ in the sector $\{z:\abs{\arg{z}} < \varphi_0\}$ for some $\varphi_0 >\frac{2\pi}{3}$. To each pole is associated a finite number of solutions to the homogeneous problem, which are outgoing in a precise sense (in view of the fact this is an ODE problem, this number will be one). The location of the poles are the scattering resonances and the corresponding solutions the resonant states.
\end{Theorem}
Before we discuss the proof of this result, let us make a few observations. Firstly, unlike in the case of Type I potentials, we only establish meromorphicity in a sector. This has subsequently been improved \cite{GalZwo} to give meromorphicity on $\C \setminus (-\infty, 0]$. Secondly, as part of our proof we define quasinormal frequencies as eigenvalues of the generator of time evolution for \eq{we} on a suitable Hilbert space, with respect to a null foliation, and show that the scattering resonances are quasinormal frequencies in this sense. Finally we are able to show that the definition of Leaver \cite{Leaver, Leaver2}, as used in many practical computations, is consistent with our definition.

As discussed above for the Type I potentials, it will be convenient to change to new coordinates $(x, t)$ such that the lines $t=\textrm{const.}$ are outgoing null rays. More concretely, we introduce:
\[
t = \tau - \frac{r}{2}, \qquad x = \frac{1}{r}.
\]
with respect to these coordinates, the equation becomes:
\ben{nullwe2}
\frac{\p}{\p x} \left( x^2 \frac{\p \psi}{\p x}\right) + \frac{\p^2 \psi}{\p x \p t} - W \psi = 0,
\een
where $W(x) = \frac{1}{x^2} V\left(\frac{1}{x} \right)$ is the transformed potential, which under our assumptions is a polynomial in $x$. After Laplace transforming in the variable $t$, we are left to consider the operator
\[
\mathcal{L}_su := \frac{d}{d x} \left( x^2 \frac{d u}{d x}\right) + s\frac{d u}{d x} - W u,
\]
on the interval $I = (0,1)$, where $u$ is assumed to satisfy Dirichlet boundary conditions at $x=1$. We wish to investigate the Fredholm properties of this operator, understood as an unbounded operator acting on a suitable Hilbert space. A key challenge in this approach will be to identify \emph{which} Hilbert spaces we should consider. 

In order to study $\mathcal{L}_s$, it will be convenient to first consider a simpler, regularised family of operators, so we consider the solvability of the family of singular ODE problems:
\begin{align}
{L}_{s, \kappa} u:= \frac{d}{dx}\left( (\kappa x + x^2) \frac{d u}{dx} \right) + s \frac{du}{dx} &= f, \label{degODE}
\end{align}
where $f \in C^\infty(\overline{I})$ is given, and $u(1)=0$. It will be convenient to denote $\mathcal{L}^0_s:=L_{s, 0}$. For $\kappa >0$, we can apply the same methods as we considered in the discussion of Type I potentials to deduce that  ${L}_{s, \kappa}: \tilde{\mathcal{D}}^k \to H^{k-1}(I)$ is a holomorphic family of Fredholm  operators for $s \in \{ z \in \mathbb{C}| \Re(z)>-\kappa(k-\frac{1}{2})\}=:\tilde{\Omega}_k$. The set of points $s\in \tilde{\Omega}_k$ at which $L_{s, \kappa}$ has non-trivial kernel are discrete, and are independent of $k$. Moreover, for $s \in \tilde{\Omega}_k$, we have $\textrm{Ker }L_{s, \kappa} \subset C^\infty(\overline{I})$. At each point $s \in \C\setminus \overline{\tilde{\Omega}_k}$, the operator ${L}_{s, \kappa}: \tilde{\mathcal{D}}^k \to H^{k-1}(I)$ has non-trivial kernel, and hence is not invertible.

It is clear from these observations that as $\kappa \to 0$, the region on which $L_{s, \kappa}$ is `nicely' invertible over $H^k$ becomes smaller and smaller. If we wish to retain any solvability in the left half-plane as $\kappa \to 0$, we must work with smooth functions. We note that for any $\kappa >0$,  and any $f \in C^\infty(\overline{I})$, a smooth solution to \eq{degODE} exists, except possibly at a discrete set of $s \in \C$ for which a non-trivial smooth solution to the homogeneous problem exists. One might hope that this statement continues to hold when $\kappa =0$, however we quickly note the following obstruction. The function $w_s(x) = e^{\frac{s}{x}} - e^s$ satisfies $\mathcal{L}^0_s w_s=0$, $u(1)=0$. For $\Re(s)<0$, $w_s(x)$ is smooth at $x=0$. Hence $\mathcal{L}^0_s$ cannot be invertible as an operator $C^\infty(\overline{I}) \to C^\infty(\overline{I})$ when $\Re(s)<0$. As a consequence we need to work with some function space which is more restrictive than $C^\infty(\overline{I})$.

One obvious way to exclude $w_s$ from the domain of $\mathcal{L}^0_s$ is to work with real analytic functions. However, aside from any aesthetic objections, this is too strong a restriction on the domain. To see this, we consider the modest goal of solving $\mathcal{L}^0_s u = 1$. Assuming a smooth solution exists, by differentiating the equation we can iteratively determine the derivatives of $u$ at $x=0$ and we find:
\ben{expansion}
u^{(n)}(0) = -\left(-\frac{1}{s}\right)^{n} n! (n-1)!
\een
for $n \geq 1$. In particular, this implies that a smooth solution $u$ to $\mathcal{L}_s u =1$ cannot have a convergent Taylor series about $x=0$ and hence cannot be real analytic. It does, however, strongly suggest that the correct regularity we should expect for $u$ is (related to) \emph{$(\sigma, 2)$-Gevrey} regularity for some $\sigma >0$. We recall:
\begin{Definition}
A function $u \in C^\infty(\overline{I})$ is $(\sigma, k)$-Gevrey regular with $k, \sigma>0$ if there exists $C>0$ such that $\sup_{x \in I} \abs{u^{(n)}(x)} \leq C \sigma^{-n} (n!)^k$ for all $n$. 
\end{Definition}
The Gevrey spaces provide a scale of spaces between $C^\infty$ and the real analytic functions. We shall not require many of their properties, but it will be convenient to note that for $k>1$ the $(\sigma, k)-$space includes bump functions and is dense in $C^\infty$.

Returning to \eq{expansion}, we see that the derivatives of $u$ at $x=0$ are consistent with:
\[
\sup_{n \in \N} \sup_{x \in I} \frac{\sigma^n}{n!^2} \abs{ u^{(n)}(x)} <\infty,
\]
for some $\sigma$, which certainly must satisfy $\sigma < \abs{s}$. On the other hand, one may establish (see Lemma \ref{expLem}) that for any $\sigma > - \Re(s)$ we have:
\[
\sup_{n \in \N} \sup_{x \in I} \frac{\sigma^n}{n!^2} \abs{ w_s^{(n)}(x)} = \infty.
\]
Combining these two facts, we can be hopeful that $\mathcal{L}^0_s$ is invertible provided the domain is defined by a suitable Gevrey-like condition, at least in a region where $\abs{s} \gg -\Re(s)$, and indeed this is the case. Having established this result, we then show that $\mathcal{L}_s$ is a compact perturbation of  $\mathcal{L}^0_s$, provided that $W$ is $(\sigma', 2)-$Gevrey regular for some $\sigma'>\sigma$. In particular this includes the polynomial potentials above, but motivates the more general:
\begin{Definition}
$V:\R_{\geq 1} \to \R$ is a type II${}_{\sigma'}$ potential if there exists $\sigma' >0$ such that the function $W:(0, 1] \to \R$ given by $W(x) := \frac{1}{x^2} V\left(\frac{1}{x} \right)$ extends to a $(\sigma', 2)-$Gevrey regular function on $[0, 1]$.
\end{Definition}
This implies in particular that:
\[
V(r) \sim \frac{1}{r^2} \sum_{k=0}^\infty \frac{V_k}{r^k} \qquad \textrm{ as } r \to \infty,
\]
for some $V_k \in \C$ with $\abs{V_k} \lesssim (\sigma')^{-k} k!$. Our main result for type II potentials is the following:
\begin{Proposition}\label{mainprop}
Suppose $V$ is a type II${}_{\sigma'}$ potential and $0<\sigma<\sigma'$. There exist $L^2-$based Gevrey spaces $X^\sigma \subset Y^\sigma$ such that the closed operator $\mathcal{L}_s: \mathcal{D}^{\sigma} \subset X^\sigma \to Y^\sigma$ is Fredholm with index $0$ for $s$ belonging to the domain $\Omega_\sigma$ illustrated in Figure \ref{OmFig}. On this domain the map $\mathcal{L}_s^{-1} : Y^\sigma \to X^\sigma$ is meromorphic, and the residues at the poles are finite rank operators. The location of the poles do not depend on $\sigma$. To each pole is associated a finite dimensional space of solutions $u \in \mathcal{D}^\sigma$ to the homogeneous problem $\mathcal{L}_s u = 0$. We call such $u$ quasinormal modes and the corresponding $s$ quasinormal frequencies.
\end{Proposition}
Implicit in this result is the fact that the domain $\mathcal{D}^{\sigma}$, defined in the usual way as the set of $u \in X^\sigma$ such that $\mathcal{L}_su \in Y^\sigma$, does not depend on $s$. In order to establish this result, we derive estimates which are uniform in $\kappa>0$ and use these to establish a Fredholm alternative for the operator with $\kappa=0$. 

\begin{figure}
\begin{overpic}[width=.5\textwidth]{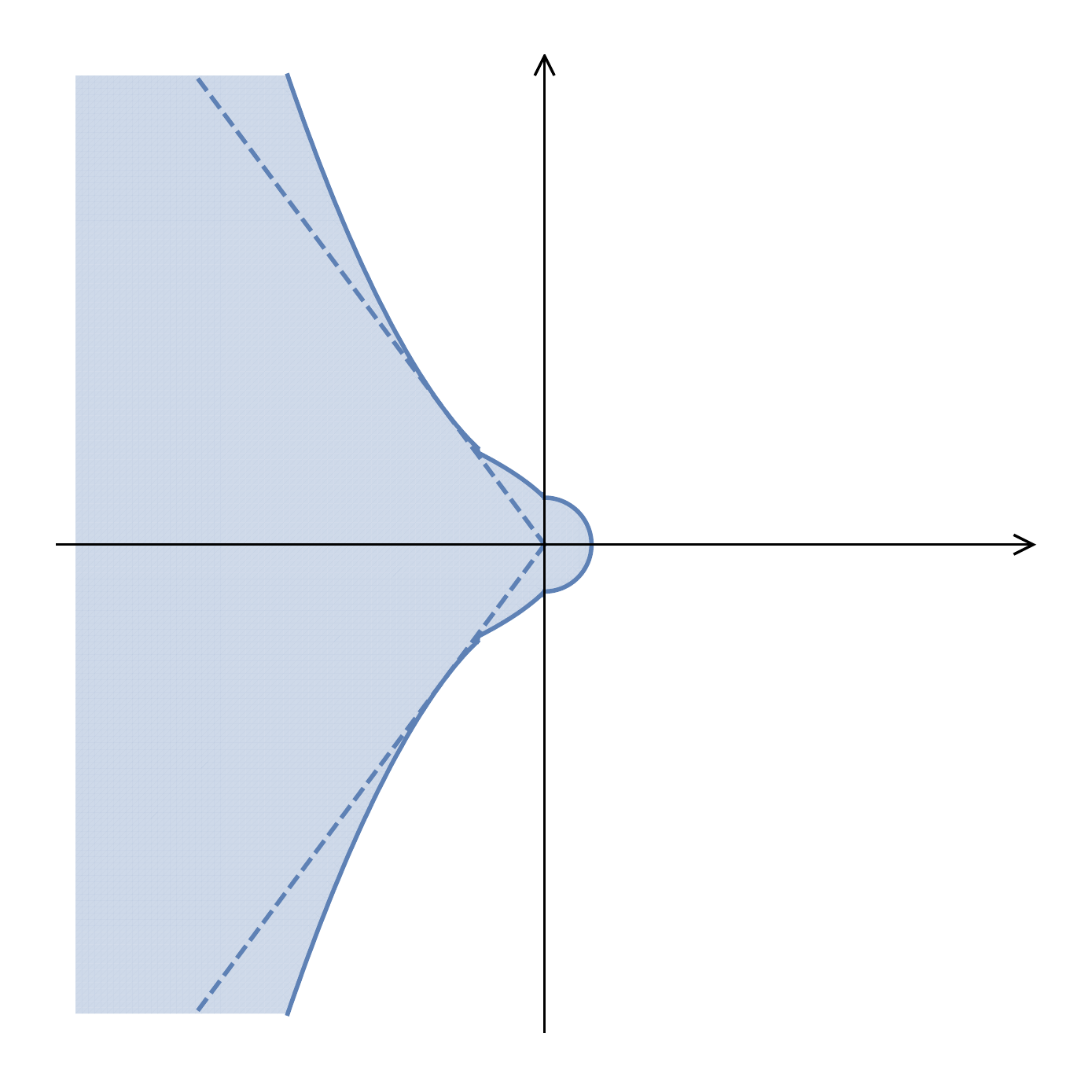}
         \put (47.5,96) {\footnotesize$\Im{s}$}
      \put (97,49) {\footnotesize$\Re{s}$}
      \put (75,75) {$\Omega_\sigma$}
	\end{overpic}
	\caption{The set $\Omega_\sigma$. As $\sigma$ varies, $\Omega_\sigma$ changes by rescaling, so that $\Omega_{\sigma} = \sigma^{-1} \Omega_1$. The dashed line indicates the boundary of the sector $\cup_{\sigma}\Omega_{\sigma} = \{\arg{s}<\varphi_0\}$, where numerically we find $\varphi_0 \simeq 0.704 \pi$.\label{OmFig}}
\end{figure}

Theorem \ref{Thm1} follows as a corollary of this result. In particular, we recover a precise characterisation of `outgoing': a solution of $\hat{L}_s w = 0$ is outgoing if $w(r) = e^{-\frac{s}{2} r} u(r^{-1})$, where $u \in X^\sigma$. In particular, this gives an asymptotic expansion:
\ben{og2}
w(r) \sim e^{-\frac{s}{2} r} \sum_{k=0}^\infty \frac{u_k}{r^k}\quad \textrm{ as }r \to \infty,
\een
for some $u_k \in \C$ satisfying $\abs{u_k} \lesssim \sigma^{-k} k!$. Note that in contrast to the situation for Type I potentials, this expansion alone does \emph{not} exclude the other branch of the general solution to $\hat{L}_s w = 0$ when $\Re(s)<0$, since $e^{\frac{s}{2} r}$ will be subleading to every term in the expansion. Thus the asymptotic expansion \eq{og2},  at least as usually understood, is not by itself a sufficient definition of outgoing for Type II potentials.

As was the case for type I potentials, we have gained something in our interpretation of the quasinormal modes. We shall see that \eq{nullwe2} naturally gives rise to a semigroup $\mathcal{S}(t): X^\sigma\to X^\sigma$, whose generator is given by:
\[
\mathcal{A} u = \int_x^1\left[ \p_\xi(\xi^2 \p_\xi u)- W(\xi) u(\xi) \right]d \xi
\]
so that $\mathcal{S}(t) = e^{\mathcal{A} t}$. The domain of $\mathcal{A}$ is precisely $\mathcal{D}^\sigma$ and moreover $u\in \mathcal{D}^\sigma$ is an eigenfunction of $\mathcal{A}$ with eigenvalue $s$ if and only if $\mathcal{L}_s u = 0$. We deduce that the quasinormal frequencies in $\Omega_\sigma$ are simply the eigenvalues of the generator of time evolution on $X^\sigma$, with respect to a null foliation.

Our assumptions on the potential do not require analyticity outside a compact set, and so approaches to define scattering resonances through complex scaling are not immediately applicable. After the appearance of this paper as a preprint, Galkowski--Zworski  \cite{GalZwo} showed that this issue can be overcome by observing that the Gevrey condition we assume on $W$ implies that the potential $V$ may be written as the sum of a piece which extends to an analytic function in a sector and a piece which decays exponentially. By treating the exponentially decaying part of the potential as a perturbation, they are able to show (in our notation) that $\hat{L}_s^{-1}: L_c^2(\R_{>1}) \to H^2_{loc.}(\R_{>1})$ admits a meromorphic extension to strips $\{s \in \C: \Re(s) > - \frac{\sigma'}{2}\}\setminus (-\infty, 0]$. This is a considerable improvement over our results in a neighbourhood of the negative real axis. 

Furthermore, their approach is also able to treat potentials $V$ which decay to leading order like $r^{-1}$. With our approach, in order that $W$ is finite at $x=0$ we are restricted to potentials decaying as $V(r) \sim r^{-2}$. It may be that this can be resolved within our framework, but it would certainly require a modified approach to establish the necessary estimates.

Galkowski--Zworski also, by a more refined estimate than our Lemma \ref{expLem}, show that the function $x \mapsto e^{\frac{s}{x}}$ is $(\sigma, 2)$-Gevrey regular on $[0,1]$ if and only if $\sigma \leq \frac{\abs{s} - \abs{\Im(s)}}{2}$. This implies in particular that if $\Omega_\sigma' \supset \Omega_\sigma$ is the maximal domain on which the conclusion of Proposition \ref{mainprop} holds, then $\Omega'_\sigma$ must look similar to $\Omega_\sigma$ as $\Im(s) \to \infty$. Indeed the boundary of $\Omega'_\sigma$ can at best be given by $\sigma \abs{\Im(s)} \sim \Re(s)^2 /4$, which is also the asymptotic form of the boundary of $\Omega_\sigma$. In this sense the region $\Omega_\sigma$ is optimal, at least as $\Im(s) \to \infty$. Near $s=0$, there appears to be room for improvement.

We note that the results of \cite{GalZwo}, which in some aspects extend those of this paper, nevertheless do not reproduce all of our results. In particular, there is no equivalent statement to Proposition \ref{mainprop}. In our view this result, which is the central one of the current work, has some advantages over meromorphic continuation results of the type appearing in Theorem \ref{Thm1} (and in \cite{GalZwo}). 

Firstly, the fact that in Theorem \ref{Thm1} one modifies the domain and range in order to continue beyond $\{\Re(s)=0\}$ means that the meromorphic continuation provides information in compact sets in $r$ about the late time behaviour of solutions of \eq{we} arising from compactly supported initial data. The result of Proposition \ref{mainprop}, however, permits us to obtain information \emph{globally} in $x$ and at late times for solutions of \eq{nullwe2} arising from data whose support is not restricted. In the context of the study of extremal black holes, this distinction is important -- requiring initial data for some wave equation on an extremal black hole background to be supported away from the future horizon is a physically restrictive assumption. Relatedly, we show that any scattering resonance (i.e.\ pole of the meromorphic continuation of $\hat{L}_s^{-1}$) is necessarily a quasinormal frequency (i.e.\ eigenvalue of $\mathcal{A}$), but not the converse. In the case of sub-extremal horizons (type I potentials) it is known that in general the scattering resonances may be a \emph{strict} subset of the quasinormal frequencies (see \cite[\S6]{CW}).

Secondly, the precise quantification of the `outgoing' condition provided by Proposition \ref{mainprop} has value, not least in the proof of Theorem \ref{LeaverThm} which establishes that Leaver's method gives correct results, at least in some region of the complex plane. Given that Leaver's method is one of the most commonly used algorithms for finding quasinormal frequencies in the physics literature this result is worthwhile in itself. 

Finally we find, purely as a matter of aesthetics, that it is satisfying to realise resonances as eigenvalues of an operator which directly arises as a generator of time translations for the original time-dependent problem. In contrast, the method of complex scaling realises resonances as eigenvalues of an operator which is connected to those arising from the time-dependent problem less directly, through an analytic continuation.

%

\section{The function spaces}
We now introduce the function spaces we shall require. We have already introduced the space of $(\sigma, k)-$Gevrey regular functions, which can be made into a Banach space in a straightforward fashion. For the majority of our results, however, we shall require $L^2-$based Gevrey spaces which have additional Hilbert space structure, and which are well adapted to the energy estimate approach we shall take. Assume $u \in C^\infty(\overline{I})$. Fix $\sigma >0$ and an integer $N$, then for $M\geq N$ an integer, $k \in \{ 0, 1, 2\}$ and $l\in \{ 0, 1 \}$ we introduce the partial semi-norm:
\be
\snorm{u}{\sigma, k, l}^{N,M} := \left( \sum_{n=N}^{M} \frac{\sigma^{2n}}{n!^2(n+1)!^2}  n^{k+l}  \int_0^1 \left( \frac{x}{\sigma}\right)^k \abs{\partial_x^n u}^2dx \right)^\frac{1}{2}
\ee
and the full Gevrey semi-norm:
\be
\snorm{u}{\sigma, k, l}^N := \left( \sum_{n=N}^{\infty} \frac{\sigma^{2n}}{n!^2(n+1)!^2}  n^{k+l}  \int_0^1 \left( \frac{x}{\sigma}\right)^k \abs{\partial_x^n u}^2dx \right)^\frac{1}{2}
\ee
For $k+l=0$ we define the norm:
\[
\norm{u}{\sigma, 0, 0} :=\snorm{u}{\sigma, 0, 0}^0,
\]
and in the case $k+l>0$ we define:
\[
\norm{u}{\sigma, k, l} :=\snorm{u}{\sigma, k, l}^0+ \left(\int_0^1 \left( \frac{x}{\sigma}\right)^k \abs{ u}^2dx\right)^{\frac{1}{2}}.
\]
We note the following useful facts:
\be
\norm{u}{\sigma, k, l} \sim  \left( \sum_{n=0}^{\infty} \frac{\sigma^{2n}}{n!^2(n+1)!^2}  (n+1)^{k+l}  \int_0^1 \left( \frac{x}{\sigma}\right)^k \abs{\partial_x^n u}^2dx \right)^\frac{1}{2}
\ee
where $\sim$ denotes equivalence of norms. We also have
\be
N^{\frac{1}{2}}  \snorm{u}{ \sigma, k, 0}^{N,M} \leq \snorm{u}{ \sigma, k, 1}^{N,M}
\ee
and
\[
\left(\snorm{u}{ \sigma, 1, l}^N\right)^2 \leq \snorm{u}{ \sigma, 0, l}^N \snorm{u}{ \sigma, 2, l}^N
\]
which follows from the Cauchy-Schwarz inequality. It will also be useful to introduce the partial boundary semi-norm:
\[
\ssnorm{u}{\sigma}^{N, M} := \left(\sum_{ n =N}^M \frac{\sigma^{2n+1}}{n!^2(n+1)!^2} \abs{u^{(n)}(0)}^2\right)^{\frac{1}{2}}
\]
and the full boundary Gevrey semi-norm:
\[
\ssnorm{u}{\sigma}^{N} := \left(\sum_{ n =N}^\infty \frac{\sigma^{2n+1}}{n!^2(n+1)!^2} \abs{u^{(n)}(0)}^2\right)^{\frac{1}{2}}.
\]
We introduce the $L^2-$Gevrey spaces:
\begin{align*}
G_{\sigma,k, l} = \left \{ u \in C^\infty(\overline{I}) \Big | \norm{u}{\sigma, k, l} < \infty \right\}.
\end{align*}
These are Hilbert spaces, with inner product defined from the norm in the obvious fashion. Finally, we introduce the spaces $X^\sigma, Y^\sigma$ by:
\[
X^\sigma := \left \{ u \in C^\infty(\overline{I}) \Big | \norm{\p_x u}{\sigma, 0, 0}+\snorm{\p_x u}{\sigma, 1, 0}^0+\snorm{\p_x u}{\sigma, 2, 0}^0 + \ssnorm{\p_x u}{\sigma}^0< \infty, u(1)=0 \right\},
\]
\[
Y^\sigma := \left \{ u \in C^\infty(\overline{I}) \Big | \norm{u}{\sigma, 0, 0}+\snorm{u}{\sigma, 1, 0}^0+ \ssnorm{ u}{\sigma}^0< \infty \right\}.
\]
Again, these are Hilbert spaces in a natural way.
\subsection{Compactness} Crucial to our argument will be the following compactness results, which are adapted from \cite[\S III.10]{Bourbaki}
\begin{Theorem}\label{compact}
The embedding $X^{\sigma} \hookrightarrow Y^\sigma$ is compact.
\end{Theorem}
\begin{proof}
We first show that the space $\tilde{G}_{\sigma, 0, 0} := \{ u \in C^\infty(\overline{I}) | \norm{\p_x u}{\sigma, 0, 0} < \infty, u(1)=0\}$ embeds compactly into $G_{\sigma, 0, 0}$. We note that:
\begin{align*}
\norm{u}{\sigma, 0, 0}^2 &= \sum_{n=0}^{\infty} \frac{\sigma^{2n}}{n!^2(n+1)!^2}   \int_0^1  \abs{\partial_x^n u}^2dx \\
&= \int_0^1 \abs{u}^2 dx + \sum_{n=0}^{\infty} \frac{\sigma^{2n+2}}{(n+1)!^2(n+2)!^2}   \int_0^1  \abs{\partial_x^{n+1} u}^2dx \\
&\leq \int_0^1 \abs{u}^2 dx + \sigma^2 \sum_{n=0}^{\infty} \frac{\sigma^{2n}}{n!^2(n+1)!^2}   \int_0^1  \abs{\partial_x^{n+1} u}^2dx\\
&\leq C_{\sigma} \norm{\partial_x u}{\sigma, 0, 0}^2
\end{align*}
where in the last line we used that $\norm{u}{L^2(I)}$ is controlled by $\norm{\p_x u}{L^2(I)}$ for a function satisfying $u(1)=0$. As a result, we see that the embedding $\tilde{G}_{\sigma, 0, 0}\hookrightarrow G_{\sigma, 0, 0}$ is continuous.

Let $B$ be the closed unit ball in $\tilde{G}_{\sigma, 0, 0}$. We wish to show that $B$ is precompact in $G_{\sigma, 0, 0}$. Fix $\epsilon>0$ and suppose $u \in B$. Then for $n\geq 0$ we have:
\[
\sum_{n=0}^\infty \frac{\sigma^{2n}}{n!^2(n+1)!^2}   \int_0^1 \abs{\partial_x^{n+1} u}^2dx \leq 1,
\]
which implies
\[
\sum_{n=p}^\infty \frac{\sigma^{2n}}{n!^2(n+1)!^2}   \int_0^1 \abs{\partial_x^{n} u}^2dx \leq \frac{\sigma^2}{p^2 (p+1)^2}  \sum_{n=p}^\infty \frac{\sigma^{2(n-1)}}{n!^2(n-1)!^2}   \int_0^1 \abs{\partial_x^{n} u}^2dx \leq \frac{\sigma^2}{p^2 (p+1)^2}.
\]
Hence, we may choose $p\in \N$ sufficiently large that $\snorm{u}{\sigma, 0, 0}^p <\epsilon$ for all $u \in B$. In particular, this implies that for $u, w \in B$ we have:
\[
\norm{w-u}{\sigma, 0, 0} \leq C_{p, \sigma} \norm{w-u}{H^p(I)} + 2 \epsilon.
\]

Now, since $\tilde{G}_{\sigma, 0, 0}\hookrightarrow H^{p+1}(I)$, we have that $B$ is bounded in $H^{p+1}(I)$, hence totally bounded in $H^{p}(I)$ by Rellich-Kondrachov. Thus there exists a finite set $C \subset B$ such that for any $u \in B$ we can find $w \in C$ with $\norm{w-u}{H^p(I)} < C_{p, \sigma}^{-1} \epsilon$. By construction, we then have:
\[
\norm{w-u}{\sigma, 0, 0} \leq 3 \epsilon.
\]
and hence $B$ is totally bounded in $G_{\sigma, 0, 0}$.

A similar argument shows that $\tilde{G}_{\sigma, 1, 0} := \{ u \in C^\infty(\overline{I}) | \norm{\p_x u}{\sigma, 1, 0} < \infty, u(1)=0\}$ embeds compactly into $G_{\sigma, 1, 0}$, the only difference being that one requires the Rellich-Kondrachov result for a weighted space, (see for example \cite{Holzegel:2012wt}). The boundary norms can be similarly dealt with (using Bolzano-Weierstrass in place of Rellich-Kondrachov) to give the result.
\end{proof}
We remark that this same basic proof can be easily adapted to establish several other compact embeddings we shall require, such as $G_{\sigma', 0, 0} \cc G_{\sigma, 0, 0}$ for $\sigma' >\sigma$, but we will not give detailed proofs each time.

We shall now consider some basic operators between our Gevrey spaces and establish boundedness and compactness as appropriate.

\begin{Theorem}\label{ProdThm}
Suppose $u \in G_{\sigma, k, l}$ and $f$ is $(\sigma', 2)-$Gevrey for some $\sigma' > \sigma$. Then $f u  \in G_{\sigma, k, l}$ and there exists a constant $C$ depending on $\sigma, \sigma', k, l$ and $f$ such that:
\[
\norm{fu}{\sigma, k, l} \leq C\norm{u}{\sigma, k, l}.
\]
\end{Theorem}
\begin{proof}
By assumption there exists $K>0$ such that:
\[
\sup_{x \in I} \abs{f^{(n)}(x)} \leq K (\sigma')^{-n} (n!)^2
\]
for all $n$. Let $w = f u$. Then by Leibniz rule we have:
\[
\frac{\p^n w}{n!} = \sum_{a+b=n} \frac{\p^a f}{a!} \frac{\p^b u}{b!}.
\]
This implies:
\begin{align*}
&\frac{\sigma^n}{n!(n+1)!} (n+1)^{\frac{1}{2}(k+l)} \abs{\p^n w} \\&\qquad  \leq \sum_{a+b=n} (\sigma')^a \frac{\abs{\p^af}}{a!^2}  \cdot\frac{\sigma^b}{b!(b+1)!} (b+1)^{\frac{1}{2}(k+l)} \abs{\p^b u}\cdot \left( \frac{\sigma}{\sigma'}\right)^a \cdot \frac{a!(b+1)!}{(n+1)!} \left(\frac{n+1}{b+1}\right)^{\frac{1}{2}(k+l)}.
\end{align*}
Now, we note that if $a+b=n$ and $n$ is sufficiently large (low $n$ terms can be trivially bounded):
\begin{align*}
\frac{a!^2(b+1)!^2}{(n+1)!^2} \left(\frac{n+1}{b+1}\right)^{k+l} &= \frac{a!(b+1)!}{(n+1)!} \cdot \frac{a!(b+1-k-l)!}{(n+1-k-l)!} \\& \quad \times  \frac{n+1}{n+1} \cdots \frac{n+1}{n-k-l+2} \cdot  \frac{b+1}{b+1} \cdots \frac{b-k-l+2}{b+1} \\
\end{align*}
The terms on the second line can be bounded uniformly in $n, b$ by a constant depending only on $k, l$, while the terms on the first line are both inverse powers of binomial coefficients, and hence bounded above by $1$. We deduce:
\begin{align*}
&\frac{\sigma^n}{n!(n+1)!} (n+1)^{\frac{1}{2}(k+l)} \abs{\p^n w} \\&\qquad  \leq C_{k, l} K \sum_{b=0}^n \frac{\sigma^b}{b!(b+1)!} (b+1)^{\frac{1}{2}(k+l)} \abs{\p^b u}\cdot \left( \frac{\sigma}{\sigma'}\right)^{n-b}.
\end{align*}
where we have used our assumption on $f$.  Now, by Cauchy-Schwarz we have:
\begin{align*}
&\frac{\sigma^{2n}}{n!^2(n+1)!^2} (n+1)^{k+l} \abs{\p^n w}^2 \\&\qquad  \leq C_{k, l, K}\left(\sum_{b=0}^n \frac{\sigma^{2b}}{b!^2(b+1)!^2} (b+1)^{k+l} \abs{\p^b u}^2\cdot \left( \frac{\sigma}{\sigma'}\right)^{n-b} \right) \left( \sum_{b=0}^n \left( \frac{\sigma}{\sigma'}\right)^{n-b}\right).
\end{align*}
Recalling that $\sigma'>\sigma$ by assumption, we see:
\[
 \sum_{b=0}^n \left( \frac{\sigma}{\sigma'}\right)^{n-b} =  \sum_{a=0}^n \left( \frac{\sigma}{\sigma'}\right)^{a} \leq \frac{\sigma'}{\sigma'-\sigma}
\]
so that:
\begin{align*}
&\sum_{n=0}^M \frac{\sigma^{2n}}{n!^2(n+1)!^2} (n+1)^{k+l} \abs{\p^n w}^2 \\&\qquad  \leq C_{k, l, K, \sigma, \sigma'}  \sum_{n=0}^M \sum_{b=0}^n \frac{\sigma^{2b}}{b!^2(b+1)!^2} (b+1)^{k+l} \abs{\p^b u}^2\cdot \left( \frac{\sigma}{\sigma'}\right)^{n-b}.
\end{align*}
Finally, since $ \sum_{n=0}^M \sum_{b=0}^n= \sum_{b=0}^M \sum_{n=b}^M$, and $\sum_{n=b}^M \left( \frac{\sigma}{\sigma'}\right)^{n-b}\leq \sum_{n=b}^\infty \left( \frac{\sigma}{\sigma'}\right)^{n-b} = \sigma'/(\sigma'-\sigma)$, we conclude:
\[
\sum_{n=0}^M \frac{\sigma^{2n}}{n!^2(n+1)!^2} (n+1)^{k+l} \abs{\p^n w}^2 \leq  C_{k, l, K, \sigma, \sigma'}\sum_{b=0}^M \frac{\sigma^{2b}}{b!^2(b+1)!^2} (b+1)^{k+l} \abs{\p^b u}^2,
\]
upon multiplying by $\left(\frac{x}{\sigma}\right)^k$ and integrating over $I$, the result follows by sending $M \to \infty$.
\end{proof}
Note that in fact our method of proof can be readily adapted to establish the result:
\[
\snorm{u}{\sigma, k, l}^{0,M} \leq C \snorm{f u}{\sigma, k, l}^{0,M},
\]
where the constant $C$ does not depend on $M$.

Combining the two theorems above, we immediately obtain:
\begin{Corollary}\label{compact2}
Suppose $f$ is $(\sigma', 2)-$Gevrey for some $\sigma'>\sigma$. Then the map
\be
\begin{array}{rcrcl}
F&:&X^\sigma&\to &Y^\sigma \\
&& u&\mapsto& fu
\end{array}
\ee
is compact.
\end{Corollary}

Next, it will be convenient to characterise first order differential operators mapping $X^\sigma \to Y^\sigma$.
\begin{Theorem}\label{DThm}
The map
\be
\begin{array}{rcrcl}
D&:&X^\sigma&\to &Y^\sigma \\
&& u&\mapsto& \frac{d u}{dx}
\end{array}
\ee
is bounded, while the map
\be
\begin{array}{rcrcl}
\tilde{{D}}&:&X^\sigma&\to &Y^\sigma \\
&& u&\mapsto& x \frac{d u}{dx}
\end{array}
\ee
is compact.
\end{Theorem}
\begin{proof}
The first part follows immediately from the definition of the spaces $X^\sigma$ and $Y^\sigma$. To establish the second part, we first claim that if $u \in X^\sigma$, then $w:=\tilde{{D}}u \in G_{\sigma, 0, 2}$. To see this, we note that:
\[
\p^n w = x \p^{n+1} u + n \p^n u.
\]
As a result,
\begin{align*}
 \sum_{n=0}^{\infty} \frac{\sigma^{2n}}{n!^2(n+1)!^2}  n^{2}  \int_0^1 \abs{\partial_x^n w}^2dx &\lesssim \sum_{n=0}^{\infty} \frac{\sigma^{2n}}{n!^2(n+1)!^2}  n^{2}  \int_0^1 \left(\frac{x}{\sigma}\right)^2\abs{\partial_x^{n+1} u}^2dx \\&\qquad + \sum_{n=0}^{\infty} \frac{\sigma^{2n}}{n!^2(n+1)!^2}  n^{4}  \int_0^1\abs{\partial_x^{n} u}^2dx \\
 &\lesssim \norm{\p u}{\sigma,0,0} + \norm{\p u}{\sigma,2,0}
\end{align*}
To complete the proof, we note that $G_{\sigma, 0, 2} \cc G_{\sigma, 0, 0}$ by an argument analogous to the proof of Theorem \ref{compact}. We also have that $G_{\sigma, 0, 2} \hookrightarrow G_{\sigma, 1, 1}$ since $\int_0^1 x \abs{u}^2 dx \leq \int_0^1 \abs{u}^2 dx$. Observing that $G_{\sigma, 1, 1} \cc G_{\sigma, 1, 0}$, and making a similar argument for the boundary terms, the result follows.
\end{proof}

\section{The time dependent problem}

As discussed in the introduction, after a suitable change of variables, we can consider the wave equation \eq{we} in the form
\ben{nullwelnew}
\frac{\p}{\p x} \left( x^2 \frac{\p \psi}{\p x}\right) + \frac{\p^2 \psi}{\p x \p t} - W(x) \psi = 0,
\een
defined on the interval $0<x\leq 1$, where we assume that $W$ is $(\sigma', 2)-$Gevrey regular for some $\sigma'$. The main result of this section is to show that \eq{nullwelnew} preserves the space $X^\sigma$ for $\sigma < \sigma'$.

\begin{Theorem}
Suppose $\psi$ solves \eq{nullwelnew} subject to Dirichlet boundary conditions at $x=1$. Then if $\psi(0, \cdot) \in X^\sigma$, we have that $\psi(t, \cdot) \in X^\sigma$ for all $t\geq 0$ and we have the estimate:
\be
\norm{\psi(t_2)}{X^\sigma} - \norm{\psi(t_1)}{X^\sigma}  \leq c\int_{t_1}^{t_2} \norm{\psi(t)}{X^\sigma} dt.
\ee
for some $c\geq 0$.
\end{Theorem}\label{GevProp}
\begin{proof}
We rewrite \eq{nullwelnew} as:
\[
 x^2 \frac{\p^2 \psi}{\p x^2}  + 2x \frac{\p \psi}{\p x} + \frac{\p^2 \psi}{\p x \p t} - W \psi = 0
 \]
Differentiating this equation $n$ times with respect to $x$, we find:
\ben{shifteddtd}
x^2 \frac{\p^2 \psi^{(n)}}{dx^2}  +  2(n+1)x \frac{\p \psi^{(n)}}{\p x} +  \frac{\p^2 \psi^{(n)}}{\p x \p t}+ n(n+1) \psi^{(n)} - \frac{\p^n (W\psi)}{\p x^n}= 0. 
\een
Where $\psi^{(n)}:= \p_x^n \psi$. We multiply this equation by $\left(1+\frac{nx}{\sigma }+ \frac{n^2x^2}{\sigma^2}\right)\p_x \overline{\psi}^{(n)}$ and take the real part:
\begin{align*}
&\left(1+\frac{nx}{\sigma}+ \frac{n^2x^2}{\sigma^2}\right) x^2 \frac{\p}{\p x} \frac{1}{2} \abs{\p_x \psi^{(n)}}^2  + 2(n+1)x\left(1+\frac{nx}{\sigma}+ \frac{n^2x^2}{\sigma^2}\right) \abs{\p_x\psi ^{(n)}}^2\\&\qquad  + \left(1+\frac{nx}{\sigma}+ \frac{n^2x^2}{\sigma^2}\right) \frac{\p}{\p t} \frac{1}{2} \abs{\p_x\psi ^{(n)}}^2\\&\qquad+ \Re\left[  \left(1+\frac{nx}{\sigma}+ \frac{n^2x^2}{\sigma^2}\right)\p_x \overline{\psi}^{(n)}\left(n(n+1){\psi}^{(n)} -\frac{\p^n (W\psi)}{\p x^n}\right) \right]= 0
\end{align*}
which we can re-write as:
\begin{align*}
& \frac{\p}{\p x} \left[ \frac{1}{2}  \left(1+\frac{nx}{\sigma}+ \frac{n^2x^2}{\sigma^2}\right)  x^2 \abs{\p_x\psi^{(n)}}^2   \right]  \\&
\qquad  + \left[ (2n+1)x +\left (2n^2 + \frac{n}{2}\right)\frac{x^2}{\sigma} + 2\frac{n^3 x^2}{\sigma^2} \right] \abs{\p_xu^{(n)}}^2\\&\qquad + \frac{\p}{\p t} \left[ \left(1+\frac{nx}{\sigma}+ \frac{n^2x^2}{\sigma^2}\right)  \frac{1}{2} \abs{\p_x\psi ^{(n)}}^2 \right] \\&\qquad = -\Re\left[  \left(1+\frac{nx}{\sigma}+ \frac{n^2x^2}{\sigma^2}\right)\p_x \overline{\psi}^{(n)}\left(n(n+1){\psi}^{(n)} -\frac{\p^n (W\psi)}{\p x^n}\right) \right]
\end{align*}
Noting that the second line is positive, we can estimate:
\begin{align}
&\nonumber \frac{\p}{\p x} \left[ \frac{1}{2}  \left(1+\frac{nx}{\sigma}+ \frac{n^2x^2}{\sigma^2}\right)  x^2 \abs{\p_x\psi^{(n)}}^2   \right]  \\&\qquad + \frac{\p}{\p t} \left[ \left(1+\frac{nx}{\sigma}+ \frac{n^2x^2}{\sigma^2}\right)  \frac{1}{2} \abs{\p_x\psi ^{(n)}}^2 \right]\label{tdest1} \\&\qquad \lesssim \left(1+\frac{nx}{\sigma}+ \frac{n^2x^2}{\sigma^2}\right)\left(\abs{\p_x \psi^{(n)}}^2+n^2(n+1)^2\abs{\psi^{(n)}}^2 + \abs{\frac{\p^n (W \psi) }{\p x^n} }^2   \right)\nonumber
\end{align}
Now, we multiply by $\frac{\sigma^{2n}}{n!^2 (n+1)!^2}$, integrate over $x\in [0,1]$ and sum over $n$ from $0$ to $\infty$. The first line gives a positive contribution after integration. The second line will give:
\[
\frac{d}{dt} \frac{1}{2} \left[  \norm{\p_x\psi}{\sigma, 0, 0}^2 + (\snorm{\p_x \psi}{\sigma, 1, 0}^0)^2 + (\snorm{\p_x \psi}{\sigma, 2, 0}^0)^2  \right].
\]
On the right-hand side, the first term is immediately controllable by $  \norm{\p_x \psi}{\sigma, 0, 0}^2 + (\snorm{\p_x \psi}{\sigma, 1, 0}^0)^2 + (\snorm{\p_x \psi}{\sigma, 2, 0}^0)^2$. For the second term we note that:
\begin{align*}
&\sum_{n=0}^\infty  \frac{\sigma^{2n}}{n!^2 (n+1)!^2}  n^2 (n+1)^2  \int_0^1 \left(1+\frac{nx}{\sigma}+ \frac{n^2x^2}{\sigma^2}\right) \abs{\psi^{(n)}}^2 \\
&\qquad =\sigma^2 \sum_{n=0}^\infty  \frac{\sigma^{2n}}{n!^2 (n+1)!^2} \int_0^1 \left(1+\frac{(n+1)x}{\sigma}+ \frac{(n+1)^2x^2}{\sigma^2}\right) \abs{\psi^{(n+1)}}^2 \\
&\qquad \lesssim \norm{\p_x \psi}{\sigma, 0, 0}^2 + (\snorm{\p_x \psi}{\sigma, 1, 0}^0)^2 + (\snorm{\p_x \psi}{\sigma, 2, 0}^0)^2 .
\end{align*}
where we have shifted the summation index in going from the first to the second line. Finally, we note that:
\begin{align*}
&\sum_{n=0}^\infty  \frac{\sigma^{2n}}{n!^2 (n+1)!^2}    \int_0^1 \left(1+\frac{nx}{\sigma}+ \frac{n^2x^2}{\sigma^2}\right) \abs{\frac{\p^n (W \psi) }{\p x^n} }^2 \\
&\qquad = \norm{W \psi}{\sigma, 0, 0}^2 + (\snorm{W \psi}{\sigma, 1, 0}^0)^2 + (\snorm{W \psi}{\sigma, 2, 0}^0)^2  \\
&\qquad \lesssim \norm{ \psi}{\sigma, 0, 0}^2 + (\snorm{ \psi}{\sigma, 1, 0}^0)^2 + (\snorm{ \psi}{\sigma, 2, 0}^0)^2 \\
&\qquad \lesssim \norm{\p_x \psi}{\sigma, 0, 0}^2 + (\snorm{\p_x \psi}{\sigma, 1, 0}^0)^2 + (\snorm{\p_x \psi}{\sigma, 2, 0}^0)^2,
\end{align*}
where we have made use of Theorem \ref{ProdThm} and further in the last line we have used the fact that $\psi$ vanishes at $x=0$ to control the lowest order terms by a Poincar\'e inequality. Putting these estimates together, we have:
\[
\frac{d}{dt} \left[  \norm{\p_x\psi}{\sigma, 0, 0}^2 + (\snorm{\p_x \psi}{\sigma, 1, 0}^0)^2 + (\snorm{\p_x \psi}{\sigma, 2, 0}^0)^2  \right] \lesssim \norm{\p_x \psi}{\sigma, 0, 0}^2 + (\snorm{\p_x \psi}{\sigma, 1, 0}^0)^2 + (\snorm{\p_x \psi}{\sigma, 2, 0}^0)^2.
\]
To complete the proof, we need to control the time derivative of the boundary part of the norm. For this, we note that evaluating the estimate \eq{tdest1} at $x=0$ gives:
\begin{align*}
 \frac{\p}{\p t} \left[  \frac{1}{2} \abs{\left. \p_x\psi ^{(n)}\right|_{x=1}}^2 \right] \lesssim \left(\abs{\left. \p_x \psi^{(n)}\right|_{x=1}}^2+n^2(n+1)^2\abs{\left.\psi^{(n)}\right|_{x=1}}^2 + \abs{\left.\frac{\p^n (W \psi) }{\p x^n}\right|_{x=1} }^2   \right)\nonumber
\end{align*}
Multiplying by  $\frac{\sigma^{2n+1}}{n!^2 (n+1)!^2}$ and summing over $n$, by a similar set of computations we can arrive at:
\[
\frac{d}{dt} (\ssnorm{\p_x \psi}{\sigma}^0)^2 \lesssim (\ssnorm{ \p_x \psi}{\sigma}^0)^2+  (\ssnorm{ \psi}{\sigma}^0)^2   \lesssim (\ssnorm{ \p_x \psi}{\sigma}^0)^2 + \norm{\psi}{\sigma, 0, 0}^2
\]
where the second inequality above follows from the fact that $\left. \psi \right|_{x=0}$ can be controlled by $\norm{\psi}{\sigma, 0, 0}$ using a trace estimate. Adding this to our previous estimate, we conclude that there exists a $c \geq 0$ such that:
\[
\frac{d}{dt} \norm{\psi}{X^\sigma}^2 \leq 2c  \norm{\psi}{X^\sigma}^2,
\]
whence the result follows.
\end{proof}
We can make a few observations about this proof. The first is that this proof is, in a sense, prototypical of the estimates that we shall subsequently establish for the elliptic problem. We also note that the fact that we multiply by $1/n!^2(n+1)!^2$ before summing is naturally forced on us by the equation. If, for example, we wished to multiply by $1/n!^2$, as would be appropriate to propagating a real-analytic norm, we would not be able to control the error term $n(n+1) \psi^{(n)} \overline{\p_x \psi^{(n)}}$. This is more than simply a failure of the estimate. Consider the case $W=0$, and let $a_n(t) = \p_x^n \psi (0,t)$. These functions obey the system of ODEs:
\[
\dot{a}_{n+1}(t) =- n(n+1) a_n(t)
\]
If we assume that $a_1(0) = 1$ and $a_n(0) = 0$ for $n>1$, we can solve this system to find:
\[
a_{n}(t) = n! (-t)^{n-1}
\]
This immediately tells us that even if the initial data away from $x=0$ is chosen to be $(\sigma, 1)-$Gevrey regular  for some $\sigma$ (i.e.\ real analytic with uniform radius of convergence $\sigma$), it will lose this regularity for all $t>\sigma^{-1}$. For $(\sigma, k)-$Gevrey regularity with $1<k<2$ we will not see finite time blow-up of the norm, however we will have super-exponential growth of the norms. The smallest value of $k$ for which we see (at worst) exponential growth, which is required in order to have a $C^0-$semigroup, is $k=2$. This is a manifestation of the Aretakis instability \cite{Aretakis:2011hc,Aretakis:2011ha,Aretakis:2012ei}, which has been extensively explored for finitely many derivatives.

As a final observation, we note that our estimate above may be thought of as arising from higher order versions of the $r^p$-estimates of \cite{DafRod}, which themselves can be thought of as what remains of the redshift effect in the $\kappa \to 0$ limit. We have thrown away some terms which will give us $x-$weighted integrated decay estimates. When studying the elliptic problem we shall need to keep track of these as we will make use of them to close our estimates, together with the fact that we can use the equation to improve $x-$weights at the expense of losing $t-$derivatives. That is to say we can control $\p_{x}\p_t\psi^{(n)}$ in terms of derivatives of $\psi$ only involving $x$ but which are multiplied by powers of $x$. In physical space this fact is not obviously useful, but after Laplace transforming it becomes valuable.

An important corollary of Theorem \ref{GevProp} follows immediately by standard results in semi-group theory \cite{Hille, RenardyRogers}.
\begin{Corollary}
For each $t\geq 0$ define an operator $\mathcal{S}(t):X^\sigma \to X^\sigma$ as follows:
\be
\mathcal{S}(t)\uppsi = \psi(t, \cdot),
\ee
where $\psi(t,x)$ is the unique solution to  \eq{nullwelnew} subject to Dirichlet boundary conditions at $x=1$ with the initial condition $\psi(0, \cdot) = \uppsi$. Then the family of operators $(\mathcal{S}(t))_{t\geq0}$ forms a $C^0-$semigroup acting on $X^\sigma$. The generator of $\mathcal{S}$ is the closed, densely defined, operator $\mathcal{A}:D(\mathcal{A}) \to X^\sigma$ defined by:
\[
\mathcal{A}u(x) :=  \int_x^1\left[ \p_\xi(\xi^2 \p_\xi u(\xi))- W(\xi) u(\xi) \right]d \xi,
\]
where $D(\mathcal{A})$ is the set of $u \in X^\sigma$ such that $\mathcal{A}u \in X^\sigma$. The resolvent $(\mathcal{A}-s)^{-1}:X^\sigma \to X^\sigma$ is well defined and holomorphic on $\Re(s)>c$, where $c$ is the constant in Theorem \ref{GevProp}. Finally, we may represent the solution  to  \eq{nullwelnew} subject to Dirichlet boundary conditions at $x=1$ with the initial condition $\psi(0, \cdot) = \uppsi$ through the Bromwich integral:
\[
\psi(t, \cdot) = \frac{1}{2\pi i}\int_{a-i \infty}^{a+i\infty} e^{st}  (\mathcal{A}-s)^{-1} \uppsi \, ds
\]
for $a>c$.
\end{Corollary}

\section{The Laplace transformed operator}

In order to establish the main proposition, the main step will be the following result
\begin{Theorem}\label{mainthm}
For any open set $\Upsilon$ compactly contained in $\Omega_{\sigma}$ there exists a bounded operator $\mathcal{K}:Y^\sigma \to Y^\sigma$, such that $(\mathcal{L}^0_{s}+\mathcal{K})^{-1}: Y^\sigma \to X^\sigma$ exists and is holomorphic in $s\in \Omega$.
\end{Theorem}
\begin{proof}
This follows immediately from Theorem \ref{L0Thm} below.
\end{proof}
Once we have established this, Proposition \ref{mainprop} follows by the analytic Fredholm theorem:
\begin{proof}[Proof of Proposition \ref{mainprop}]
By a standard computation, for $f \in Y^\sigma$, $\mathcal{L}^0_s u = f$ has a solution $u \in X^\sigma$ if and only if $(\iota - (\mathcal{L}^0_s + \mathcal{K})^{-1} \mathcal{K})u = (\mathcal{L}^0_s + \mathcal{K})^{-1}f$. Let $E:= (\mathcal{L}^0_s + \mathcal{K})^{-1} \mathcal{K}$. By assumption $E:Y^\sigma \to Y^\sigma$ is holomorphic on $\Upsilon$ and compact, since it maps $Y^\sigma$ to $X^\sigma$ and $X^\sigma \cc Y^\sigma$  by Theorem \ref{compact} so the result follows for $\mathcal{L}^0_s$ by the analytic Fredholm theorem \cite{RenardyRogers}.  By Corollary \ref{compact2}, $\mathcal{L}_s$ is a compact perturbation of $\mathcal{L}^0_s$, and again standard Fredholm theory provides the concusion.
\end{proof}

The choice of $\mathcal{K}$ we shall make in order to prove Theorem \ref{mainthm} is a little unusual. We shall take:
\[
\mathcal{K}u(x) := - N (N+1)u(x) + \lambda \sum_{i=0}^{N-1} \frac{(x-1)^i}{i!} u^{(i+1)}(1).
\]
Here $N \in \N$ and $\lambda \in \R$ are to be chosen later to be sufficiently large, depending on the set $\Upsilon$.  We shall establish the invertibility of $\mathcal{L}^0_s + \mathcal{K}$ by considering the $\kappa \to 0^+$ limit of $L_{s, \kappa} + \mathcal{K}$. As a consequence, we begin by obtaining estimates for solutions to the equation:
\ben{shiftedeq}
\frac{d}{dx}\left( (\kappa x + x^2) \frac{d u}{dx} \right) + s \frac{du}{dx} -N(N+1) u + \lambda \sum_{i=0}^{N-1} \frac{(x-1)^i}{i!} u^{(i+1)}(1)= f, \qquad u(1) = 0,
\een
which are uniform in $\kappa \geq 0$.

\begin{Lemma}\label{GMest} Suppose that $\kappa \geq 0$, $N, N', M$ are integers with $0\leq N\leq N'\leq M$, $\sigma >0$,  and that $u \in C^\infty(\overline{I})$ satisfies \eq{shiftedeq}. Then we have the estimate:
\begin{align*}
&\kappa \left[ \left( \snorm{\p_x u}{\sigma, 0, 1}^{N', M}\right)^2 + \left( \snorm{\p_x u}{\sigma, 1, 1}^{N', M}\right)^2 \right] + \left(2 \sigma+\Re(s) \right) \left( \snorm{\p_x u}{\sigma, 1, 0}^{N', M}\right)^2 + 2\sigma \left( \snorm{\p_x u}{\sigma, 2, 0}^{N', M}\right)^2  \\
&\qquad \leq \left( \textstyle\frac{\sigma}{2(N'+1)} -\Re(s)\right)\left(\snorm{\p_x u}{\sigma, 0, 0}^{N', M} \right)^2 +{\textstyle \frac{\sigma}{2 N' (N'+1)} }\left( \ssnorm{\p_x u}{\sigma}^{N', M} \right)^2 \\
&\quad \qquad + \left( \snorm{\p_x u}{\sigma, 0, 0}^{N', M} + \snorm{\p_x u}{\sigma, 1, 0}^{N', M}\right) \left( \snorm{f}{\sigma, 0, 0}^{N', M} + \snorm{f}{\sigma, 1, 0}^{N', M}\right) \\
&\qquad \qquad + {\textstyle\frac{\sigma^{2N'}}{N'!^2 (N'+1)!^2} \frac{N'}{2 \sigma} \left[  N'(N'+1)-N(N+1)  \right] }\int_0^1 \abs{u^{(N')}}^2 dx \\
&\qquad \qquad +  \frac{1}{2}  {\textstyle\frac{\sigma^{2N'}}{N'!^2 (N'+1)!^2}} \left[  N'(N'+1)-N(N+1)  \right]  \abs{u^{(N')}(0)}^2
\end{align*}
\end{Lemma}
\begin{proof}
We rewrite \eq{shiftedeq} as:
\[
(\kappa x + x^2) \frac{d^2 u}{dx^2}  + (\kappa + 2x +s) \frac{du}{dx} -N(N+1) u+ \lambda \sum_{i=0}^{N-1} \frac{(x-1)^i}{i!} u^{(i+1)}(1)  = f 
\]
Differentiating this equation $n\geq N$ times with respect to $x$, we find:
\ben{shiftedd}
(\kappa x + x^2) \frac{d^2 u^{(n)}}{dx^2}  + (\kappa(n+1) + 2(n+1)x +s) \frac{du^{(n)}}{dx} + \left(n(n+1)-N(N+1)\right) u^{(n)}= f^{(n)}. 
\een
We multiply this equation by $\left(1+\frac{nx}{\sigma}\right)\p_x \overline{u}^{(n)}$ and take the real part:
\begin{align*}
&\left(1+\frac{nx}{\sigma}\right)(\kappa x + x^2) \frac{d}{dx} \frac{1}{2} \abs{\p_xu^{(n)}}^2  + \left(1+\frac{nx}{\sigma}\right)(\kappa(n+1) + 2(n+1)x +\Re(s)) \abs{\p_xu^{(n)}}^2\\& + (n(n+1)-N(N+1)) \left(1+\frac{nx}{\sigma}\right)  \frac{d}{dx} \frac{1}{2} \abs{u^{(n)}}^2 =   \left(1+\frac{nx}{\sigma}\right) \Re(f^{(n)}\p_x \overline{u}^{(n)})
\end{align*}
which we can re-write as:
\begin{align*}
& \frac{d}{dx} \left[ \frac{1}{2}  \left(1+\frac{nx}{\sigma}\right) (\kappa x + x^2) \abs{\p_xu^{(n)}}^2 + \frac{1}{2} (n(n+1)-N(N+1)) \abs{u^{(n)}}^2 \left(1+\frac{nx}{\sigma}\right)   \right] \\&
\qquad + \kappa \left[ n+\frac{1}{2} + \frac{n^2 x}{\sigma} \right] \abs{\p_xu^{(n)}}^2 + \left[ (2n+1)x +\left (2n^2 + \frac{n}{2}\right)\frac{x^2}{\sigma} \right] \abs{\p_xu^{(n)}}^2 \\&\qquad = \frac{n}{2 \sigma} \left[  n(n+1)-N(N+1)  \right] \abs{u^{(n)}}^2  - \left(1+\frac{nx}{\sigma}\right) \Re(s) \abs{\p_xu^{(n)}}^2 \\
&\qquad + \left(1+\frac{nx}{\sigma}\right)\Re(f^{(n)}\p_x \overline{u}^{(n)})
\end{align*}
Integrating this identity in $x$ over the interval $I$, and assuming that $n \geq N$, we deduce:
\begin{align}
& \int_0^1 \left[ \kappa \left( n + \frac{n^2 x}{\sigma} \right) + 2 \left (nx + \frac{n^2 x^2}{\sigma} \right)  \right] \abs{\p_xu^{(n)}}^2 dx\nonumber \\&\qquad \qquad \leq  \frac{n}{2 \sigma} \left[  n(n+1)-N(N+1)  \right] \int_0^1 \abs{u^{(n)}}^2 dx\nonumber  \\
& \qquad \qquad  \quad  - \Re(s) \int_0^1 \left[ 1  + \frac{n x}{\sigma}  \right] \abs{\p_xu^{(n)}}^2 dx \label{ineq1} \\
& \qquad \qquad  \quad +\left( \int_0^1 \left[ 1  + \frac{n x}{\sigma}  \right] \abs{\p_xu^{(n)}}^2 dx\right)^\frac{1}{2}\left( \int_0^1 \left[ 1  + \frac{n x}{\sigma}  \right] \abs{f^{(n)}}^2 dx\right)^\frac{1}{2}\nonumber \\
& \qquad \qquad  \quad + \frac{1}{2}  \left[  n(n+1)-N(N+1)  \right]  \abs{u^{(n)}(0)}^2 \nonumber
\end{align}
We assume that $N'\leq n \leq M$ and we multiply by $\frac{\sigma^{2n}}{n!^2 (n+1)!^2}$, sum over $n$ from $N'$ to $M$, then bound the terms on the right hand side one at a time. Firstly:
\begin{align*}
&\sum_{n=N'}^M{\textstyle\frac{\sigma^{2n}}{n!^2 (n+1)!^2} \frac{n}{2 \sigma} \left[  n(n+1)-N(N+1)  \right] }\int_0^1 \abs{u^{(n)}}^2 dx\\&\qquad \qquad =\sum_{n=N'}^M {\textstyle\frac{\sigma}{2} \left( \frac{1}{n+1} - \frac{N(N+1)}{n(n+1)^2}  \right) {\textstyle\frac{\sigma^{2(n-1)}}{n!^2 (n-1)!^2} } }\int_0^1 \abs{u^{(n)}}^2 dx \\
&\qquad \qquad \leq \frac{\sigma}{2(N+1)} \left(\snorm{\p_x u}{\sigma, 0, 0}^{N', M} \right)^2 \\&\qquad \qquad \quad+ {\textstyle\frac{\sigma^{2N'}}{N'!^2 (N'+1)!^2} \frac{N'}{2 \sigma} \left[  N'(N'+1)-N(N+1)  \right] }\int_0^1 \abs{u^{(N')}}^2 dx
\end{align*}
Continuing, we see:
\begin{align*}
& - \Re(s) \sum_{n=N'}^M{\textstyle\frac{\sigma^{2n}}{n!^2 (n+1)!^2} } \int_0^1 \left[ 1  + \frac{n x}{\sigma}  \right] \abs{\p_xu^{(n)}}^2 dx\\&\qquad  \leq -\Re(s) \left(\snorm{\p_x u}{\sigma, 0, 0}^{N', M} \right)^2 -\Re(s) \left(\snorm{\p_x u}{\sigma, 1, 0}^{N', M} \right)^2
\end{align*}
and similarly:
\begin{align*}
&\sum_{n=N'}^M {\textstyle\frac{\sigma^{2n}}{n!^2 (n+1)!^2} } \left( \int_0^1 \left[ 1  + \frac{n x}{\sigma}  \right] \abs{\p_xu^{(n)}}^2 dx\right)^\frac{1}{2}\left( \int_0^1 \left[ 1  + \frac{n x}{\sigma}  \right] \abs{f^{(n)}}^2 dx\right)^\frac{1}{2} \\&\qquad \leq \left( \snorm{\p_x u}{\sigma, 0, 0}^{N', M} + \snorm{\p_x u}{\sigma, 1, 0}^{N', M}\right) \left( \snorm{f}{\sigma, 0, 0}^{N', M} + \snorm{f}{\sigma, 1, 0}^{N', M}\right)
\end{align*}
Finally, the boundary term can be estimated by:
\begin{align*}
&\frac{1}{2}  \sum_{n=N'}^M {\textstyle\frac{\sigma^{2n}}{n!^2 (n+1)!^2}} \left[  n(n+1)-N(N+1)  \right]  \abs{u^{(n)}(0)}^2\\ &\qquad \qquad  = \sum_{n=N'}^M{\textstyle \frac{\sigma}{2}  \left( \frac{1}{n(n+1)} - \frac{N(N+1)}{n^2(n+1)^2)}\right) \frac{\sigma^{2n-1}}{n!^2 (n-1)!^2}}\abs{u^{(n)}(0)}^2 \\
 &\qquad \qquad  \leq \frac{\sigma}{2 N' (N'+1)}  \left( \ssnorm{\p_xu}{\sigma}^{N', M} \right)^2 \\
 &\qquad \qquad  \qquad +  \frac{1}{2}  {\textstyle\frac{\sigma^{2N'}}{N'!^2 (N'+1)!^2}} \left[  N'(N'+1)-N(N+1)  \right]  \abs{u^{(N')}(0)}^2 
\end{align*}
\end{proof}

Next, we gain some control  of the boundary values of $u$ at $x=0$ in terms of the data. We first define:
\[
\Omega^1_{\sigma} := \{ s\in \C | 0<\sigma <\abs{\Im(s)}, \Re(s) \leq 0\} \cup  \{ s\in \C | 0<\sigma <\abs{s}, \Re(s) > 0\}.
\]
\begin{Lemma} \label{Best}
Suppose that $\kappa \geq 0$, $0\leq N\leq M$ are integers, $s \in \Omega^1_{\sigma}$,  and that $u \in C^\infty(\overline{I})$ satisfies \eq{shiftedeq}. Then we have the estimate:
\be
\ssnorm{\p_xu}{\sigma}^{N, M}  \leq \frac{1}{C_{s, \sigma}}  \ssnorm{f}{\sigma}^{N, M}.
\ee
Where $C_{s, \sigma} = \textrm{min}\left\{\abs{\Im(s)}-\sigma, \abs{s} - \sigma \right\}$.
\end{Lemma}
\begin{proof}
We first assume $0<\sigma <\abs{\Im(s)}, \Re(s) \leq 0$. Returning to \eq{shiftedd}, we note that the equation simplifies at $x=0$:
\[
\left( \kappa (n+1) + s\right) \p_x u^{(n)} (0) + \left[ n (n+1) - N (N+1)\right]u^{(n)} (0) = f^{(n)}(0).
\]
Multiplying by $\overline{\p_x u^{(n)}(0)}$ and taking the imaginary part yields:
\[
\abs{\Im(s)} \abs{\p_x u^{(n)} (0)}^2  \leq \left[ n (n+1) - N (N+1)\right] \abs{u^{(n)} (0)}\abs{\p_x u^{(n)} (0)} + \abs{\p_x u^{(n)} (0)}\abs{f^{(n)} (0)}
\]
We multiply by $\frac{\sigma^{2n+1}}{n!^2 (n+1)!^2}$, sum over $n$ from $N$ to $M$, then bound the terms on the right hand side, assuming $N\leq n \leq M$. We have:
\begin{align*}
&\sum_{n=N}^M{\textstyle \frac{\sigma^{2n+1}}{n!^2 (n+1)!^2} }\left[ n (n+1) - N (N+1)\right] \abs{u^{(n)} (0)}\abs{\p_x u^{(n)} (0)}\\ \qquad &=\sum_{n=N}^M \left( \textstyle\frac{\sigma^{n+\frac{1}{2}}}{n! (n+1)! } \abs{\p_x u^{(n)} (0)} \right)  \left( \textstyle \frac{\sigma^{n-\frac{1}{2}}}{(n-1)! n! } \abs{u^{(n)} (0)} \right)  \sigma \left( 1- \frac{N(N+1)}{n(n+1)} \right) \\
&\leq \sigma \left( \ssnorm{\p_xu}{\sigma}^{N, M} \right)^2
\end{align*}
Similarly,
\[
\sum_{n=N}^M{\textstyle \frac{\sigma^{2n+1}}{n!^2 (n+1)!^2} } \abs{\p_x u^{(n)} (0)}\abs{f^{(n)} (0)} \leq  \ssnorm{\p_xu}{\sigma}^{N, M}  \ssnorm{f}{\sigma}^{N, M}
\]
We deduce that:
\[
\abs{\Im(s)} \left( \ssnorm{\p_xu}{\sigma}^{N, M} \right)^2   \leq \sigma \left( \ssnorm{\p_xu}{\sigma}^{N, M} \right)^2+\ssnorm{\p_xu}{\sigma}^{N, M}  \ssnorm{f}{\sigma}^{N, M}.
\]
and hence, we conclude that if $\abs{\Im(s)} -\sigma >0$, then:
\[
\ssnorm{\p_xu}{\sigma}^{N, M}  \leq \frac{1}{\abs{\Im(s)} -\sigma}  \ssnorm{f}{\sigma}^{N, M}.
\]
The proof for $0<\Re(s)$ and $0<\sigma <\abs{s}$ proceeds almost identically, but we use as multiplier $\overline{s \p_x u^{(n)}(0)}$, take the real part and note that the term proportional to $\kappa$ has a good sign relative to the $\abs{s}^2 \abs{\p_x u^{(n)}(0)}^2$ term.
\end{proof}

These estimates are already enough to establish that for $\kappa >0$ the operator is Gevrey hypoelliptic, i.e.\ if $f$ belongs to a suitable Gevrey space, then so will $u$. 

\begin{Theorem}\label{MorThm}
Suppose $\kappa >0$, $s \in \Omega^1_\sigma$, and suppose that $u \in C^\infty(\overline{I})$ satisfies \eq{shiftedeq}.  Suppose further that:
\[
\snorm{f}{\sigma, 0, 0}^N+\snorm{f}{\sigma, 1, 0}^N+ \ssnorm{f}{\sigma}^N <\infty.
\]
Then
\[
 \snorm{\p_x u}{\sigma, 0, 0}^N+\snorm{\p_x u}{\sigma, 1, 0}^N+\snorm{\p_x u}{\sigma, 2, 0}^N + \ssnorm{\p_x u}{\sigma}^N <\infty,
\]
and we have the estimate:
\begin{align*}
&\kappa \left[ \left( \snorm{\p_x u}{\sigma, 0, 1}^{N}\right)^2 + \left( \snorm{\p_x u}{\sigma, 1, 1}^{N}\right)^2 \right] + \left(2 \sigma+\Re(s) \right) \left( \snorm{\p_x u}{\sigma, 1, 0}^{N}\right)^2 + 2\sigma \left( \snorm{\p_x u}{\sigma, 2, 0}^{N}\right)^2  \\
&\qquad \leq \left( \textstyle\frac{\sigma}{2(N+1)} -\Re(s)\right)\left(\snorm{\p_x u}{\sigma, 0, 0}^{N} \right)^2 +{\textstyle \frac{\sigma}{2 N (N+1)C_{s, \sigma}^2} }\left( \ssnorm{f}{\sigma}^{N} \right)^2 \\
&\quad \qquad + \left( \snorm{\p_x u}{\sigma, 0, 0}^{N} + \snorm{\p_x u}{\sigma, 1, 0}^{N}\right) \left( \snorm{f}{\sigma, 0, 0}^{N} + \snorm{f}{\sigma, 1, 0}^{N}\right) 
\end{align*}
where $C_{s, \sigma}$ is as in Lemma \ref{Best}.
\end{Theorem}
\begin{proof}
Combining the estimate of Lemma \ref{GMest} with that of Lemma \ref{Best} and recalling that:
\be
N'{}^{\frac{1}{2}}  \snorm{u}{ \sigma, k, 0}^{N',M} \leq \snorm{u}{ \sigma, k, 1}^{N',M}, 
\ee
we have that for any $N\leq N' \leq M$:
\begin{align*}
&\left( N'\kappa-\textstyle\frac{\sigma}{2(N'+1)} +\Re(s)-1 \right)\left( \snorm{\p_x u}{\sigma, 0, 0}^{N', M}\right)^2 + \left(N'\kappa +2 \sigma+\Re(s) -1\right) \left( \snorm{\p_x u}{\sigma, 1, 0}^{N', M}\right)^2 \\
&\qquad  + 2\sigma \left( \snorm{\p_x u}{\sigma, 2, 0}^{N', M}\right)^2  \\
&\qquad \leq {\textstyle \frac{\sigma}{2 N' (N'+1)} (\abs{\Im(s)}-\sigma)^2 }\left( \ssnorm{f}{\sigma}^{N', M} \right)^2+ \left( \snorm{f}{\sigma, 0, 0}^{N', M} \right)^2 + \left(\snorm{f}{\sigma, 1, 0}^{N', M}\right)^2 \\
&\qquad \qquad + {\textstyle\frac{\sigma^{2N'}}{N'!^2 (N'+1)!^2} \frac{N'}{2 \sigma} \left[  N'(N'+1)-N(N+1)  \right] }\int_0^1 \abs{u^{(N')}}^2 dx \\
&\qquad \qquad +  \frac{1}{2}  {\textstyle\frac{\sigma^{2N'}}{N'!^2 (N'+1)!^2}} \left[  N'(N'+1)-N(N+1)  \right]  \abs{u^{(N')}(0)}^2
\end{align*}
Clearly, fixing $N'$ sufficiently large we have:
\begin{align*}
 \snorm{\p_x u}{\sigma, 0, 0}^{N', M}+\snorm{\p_x u}{\sigma, 1, 0}^{N', M}+\snorm{\p_x u}{\sigma, 2, 0}^{N', M}  &\lesssim \snorm{f}{\sigma, 0, 0}^{N', M}+\snorm{f}{\sigma, 1, 0}^{N', M}+ \ssnorm{f}{\sigma}^{N, M} \\&\qquad + \norm{u^{(N')}}{L^2(I)}^2 + \abs{u^{(N')}(0)}^2
\end{align*}
where the implicit constant depends on $N, N', \sigma, s$, but crucially \emph{not} on $M$.  Sending $M$ to infinity, we conclude:
\[
 \snorm{\p_x u}{\sigma, 0, 0}^{N'}+\snorm{\p_x u}{\sigma, 1, 0}^{N'}+\snorm{\p_x u}{\sigma, 2, 0}^{N'} <\infty.
\]
Since we assume $u \in C^\infty(\overline{I})$, it is immediate that the statement holds with $N'$ replaced by $N$. The fact that $\ssnorm{\p_x u}{\sigma}^N$ is finite follows immediately from Lemma \ref{Best}. Returning to the estimate of Lemma \ref{GMest} and using Lemma \ref{Best} to control the boundary terms, we set $N'=N$ and send $M \to \infty$ to recover the estimate above.
\end{proof}

If we assume that $N$ is large and that $\Re(s)$ is small, our earlier estimates allow us to control $\snorm{\p_x u}{\sigma, 1, 0}^N$ and $\snorm{\p_x u}{\sigma, 1, 0}^N$ by a small multiple $\snorm{\p_x u}{\sigma, 0, 0}^N$, plus norms of the data. In order to close the estimate, we need to control $\snorm{\p_x u}{\sigma, 0, 0}^N$ by $\snorm{\p_x u}{\sigma, 1, 0}^N$ and $\snorm{\p_x u}{\sigma, 1, 0}^N$ with coefficients that are not too large. To do this we again use a multiplier estimate, but this time one which does not close at any finite order of derivatives.

\begin{Theorem} \label{SecEst}
Suppose $\kappa \geq 0$ and that $u \in C^\infty(\overline{I})$ satisfies \eq{shiftedeq}.  Suppose further that:
\[
\snorm{f}{\sigma, 0, 0}^N+\snorm{f}{\sigma, 1, 0}^N+ \ssnorm{f}{\sigma}^N <\infty.
\]
and
\[
 \snorm{\p_x u}{\sigma, 0, 0}^N+\snorm{\p_x u}{\sigma, 1, 0}^N+\snorm{\p_x u}{\sigma, 2, 0}^N + \ssnorm{\p_x u}{\sigma}^N <\infty.
\]
Then we have the estimate:
\begin{align*}
(\abs{s} - \sigma) \left(\snorm{\p_x u}{\sigma, 0, 0}^N \right)^2 &\leq \frac{-\Re(s)}{\abs{s}}\left(1+\frac{1}{N} \right)\left[\kappa \left(\snorm{\p_x u}{\sigma, 0, 1}^N \right)^2 + 2\sigma \left(\snorm{\p_x u}{\sigma, 1, 0}^N \right)^2  \right] \\
& \qquad + \left(1+\frac{2}{N} \right) \left[ \kappa\left(\snorm{\p_x u}{\sigma, 1, 1}^N \right)^2 +\sigma \left(\snorm{\p_x u}{\sigma, 2, 0}^N \right)^2 \right] \\
&\qquad + \snorm{\p_x u}{\sigma, 0, 0}^N \snorm{f}{\sigma, 0, 0}^N 
\end{align*}
\end{Theorem}
\begin{proof}
We return to the differentiated equation in the form \eq{shiftedd} and rearrange in the form:
\[
s \p_x u^{(n)} = - \left( \kappa(n+1) + 2 (n+1) x\right) \p_x u^{(n)} + (N(N+1) - n(n+1)) u^{(n)} - (\kappa x + x^2) \p^2_{x} u^{(n)} + f^{(n)}
\]
Multiplying by $\overline{s} \overline{\p_xu^{(n)}}$, taking the real part and integrating over $I$, we find:
\begin{align*}
\abs{s}^2 \int_0^1 \abs{\p_x u^{(n)}}^2 dx &\leq -\Re(s) \int_0^1 \left( \kappa(n+1) + 2 (n+1) x\right) \abs{\p_x u^{(n)}}^2 dx \\
&\qquad + \abs{s} (n(n+1) - N(N+1)) \left( \int_0^1 \abs{u^{(n)}}^2 dx\right)^\frac{1}{2}\left( \int_0^1 \abs{\p_x u^{(n)}}^2dx \right)^\frac{1}{2} \\
&\qquad + \kappa \abs{s} \left( \int_0^1 x \abs{\p_x^2 u^{(n)}}^2 \right)^\frac{1}{2} \left( \int_0^1 x \abs{\p_xu^{(n)}}^2 \right)^\frac{1}{2} \\
&\qquad + \abs{s}  \left( \int_0^1 x^2 \abs{\p_x^2 u^{(n)}}^2 \right)^\frac{1}{2} \left( \int_0^1 x^2 \abs{\p_xu^{(n)}}^2 \right)^\frac{1}{2} \\
&\qquad + \abs{s}  \left( \int_0^1 \abs{\p_x u^{(n)}}^2 \right)^\frac{1}{2} \left( \int_0^1 \abs{f^{(n)}}^2 \right)^\frac{1}{2}
\end{align*} 
We multiply by $\frac{1}{\abs{s}}\frac{\sigma^{2n}}{n!^2 (n+1)!^2}$, sum over $n\geq N$ and then bound the terms on the right hand side. Again, working one at a time we find:
\begin{align*}
&-\Re(s)\sum_{n=N}^\infty {\textstyle{\frac{1}{\abs{s}}\frac{\sigma^{2n}}{n!^2 (n+1)!^2}} }\int_0^1 \left( \kappa(n+1) + 2 (n+1) x\right) \abs{\p_x u^{(n)}}^2 dx  \\ & \qquad = -\Re(s)\sum_{n=N}^\infty \left( 1+\frac{1}{n}\right) {\textstyle{\frac{1}{\abs{s}}\frac{\sigma^{2n}}{n!^2 (n+1)!^2}} }\int_0^1 \left( \kappa n + 2 n x\right) \abs{\p_x u^{(n)}}^2 dx \\
&\qquad \leq  \frac{-\Re(s)}{\abs{s}}\left(1+\frac{1}{N} \right)\left[\kappa \left(\snorm{\p_x u}{\sigma, 0, 1}^N \right)^2 + 2\sigma \left(\snorm{\p_x u}{\sigma, 1, 0}^N \right)^2  \right] 
\end{align*}
Next,
\begin{align*}
& \abs{s} \sum_{n=N}^\infty  {\textstyle{\frac{1}{\abs{s}}\frac{\sigma^{2n}}{n!^2 (n+1)!^2}} } (n(n+1) - N(N+1)) \left( \int_0^1 \abs{u^{(n)}}^2 dx\right)^\frac{1}{2}\left( \int_0^1 \abs{\p_x u^{(n)}}^2dx \right)^\frac{1}{2}  \\
&\qquad = \sigma \sum_{n=N}^\infty \left(1-{\textstyle \frac{N(N+1)}{n(n+1)} }\right) \left(  {\textstyle{\frac{\sigma^{2(n-1)}}{(n-1)!^2 n!^2}} } \int_0^1 \abs{u^{(n)}}^2 dx \right)^\frac{1}{2} \left(  {\textstyle{\frac{\sigma^{2n}}{n!^2 (n+1)!^2}} } \int_0^1 \abs{\p_x u^{(n)}}^2 dx \right)^\frac{1}{2} \\
&\qquad \leq \sigma  \left(\snorm{\p_x u}{\sigma, 0, 0}^N \right)^2.
\end{align*}
Continuing,
\begin{align*}
&\kappa \abs{s} \sum_{n=N}^\infty  {\textstyle{\frac{1}{\abs{s}}\frac{\sigma^{2n}}{n!^2 (n+1)!^2}} }  \left( \int_0^1 x \abs{\p_x^2 u^{(n)}}^2 \right)^\frac{1}{2} \left( \int_0^1 x \abs{\p_xu^{(n)}}^2 \right)^\frac{1}{2} \\
&\qquad = \kappa \sum_{n=N}^\infty \left( \frac{n+2}{n}\right) \left({\textstyle \frac{\sigma^{2(n+1)}(n+1)^2}{(n+1)!^2 (n+2)!^2}}\int_0^1 \frac{x}{\sigma} \abs{\p_x^2 u^{(n)}}^2 \right)^\frac{1}{2} \left( {\textstyle \frac{\sigma^{2n}n^2}{n!^2 (n+1)!^2}} \int_0^1 \frac{x}{\sigma} \abs{\p_xu^{(n)}}^2 \right)^\frac{1}{2} \\
&\qquad \leq \kappa \left(1+\frac{2}{N} \right) \left(\snorm{\p_x u}{\sigma, 1, 1}^N \right)^2
\end{align*}
Similarly,
\begin{align*}
&\abs{s} \sum_{n=N}^\infty  {\textstyle{\frac{1}{\abs{s}}\frac{\sigma^{2n}}{n!^2 (n+1)!^2}} }  \left( \int_0^1 x^2 \abs{\p_x^2 u^{(n)}}^2 \right)^\frac{1}{2} \left( \int_0^1 x^2 \abs{\p_xu^{(n)}}^2 \right)^\frac{1}{2} \\
&\qquad = \sigma\sum_{n=N}^\infty  \left( \frac{n+2}{n}\right) \left({\textstyle \frac{\sigma^{2(n+1)}(n+1)^2}{(n+1)!^2 (n+2)!^2}}\int_0^1 \left(\frac{x}{\sigma}\right)^2 \abs{\p_x^2 u^{(n)}}^2 \right)^\frac{1}{2} \left( {\textstyle \frac{\sigma^{2n}n^2}{n!^2 (n+1)!^2}} \int_0^1 \left(\frac{x}{\sigma}\right)^2 \abs{\p_xu^{(n)}}^2 \right)^\frac{1}{2} \\
&\qquad \leq \sigma \left(1+\frac{2}{N} \right) \left(\snorm{\p_x u}{\sigma, 2, 0}^N \right)^2
\end{align*}
Finally,
\begin{align*}
&\abs{s}\sum_{n=N}^\infty  {\textstyle{\frac{1}{\abs{s}}\frac{\sigma^{2n}}{n!^2 (n+1)!^2}} }  \left( \int_0^1 \abs{\p_x u^{(n)}}^2 \right)^\frac{1}{2} \left( \int_0^1 \abs{f^{(n)}}^2 \right)^\frac{1}{2} \\
&\qquad =\sum_{n=N}^\infty  \left({\textstyle{\frac{1}{\abs{s}}\frac{\sigma^{2n}}{n!^2 (n+1)!^2}} }  \int_0^1 \abs{\p_x u^{(n)}}^2 \right)^\frac{1}{2} \left( {\textstyle{\frac{1}{\abs{s}}\frac{\sigma^{2n}}{n!^2 (n+1)!^2}} } \int_0^1 \abs{f^{(n)}}^2 \right)^\frac{1}{2} \\
&\leq \snorm{\p_x u}{\sigma, 0, 0}^N \snorm{f}{\sigma, 0, 0}^N.
\end{align*}
Combining all of these estimates yields the theorem.
\end{proof}
We are now, finally, in a position to close the estimates for $\kappa >0$, whence we will construct (unique!) solutions when $\kappa = 0$. The key point is that the estimate of Theorem \ref{MorThm} controls the error terms in Theorem \ref{SecEst} in terms of $\snorm{\p_x u}{\sigma, 0, 0}^N $, but with a coefficient that is small enough to absorb, at least in a certain region of the complex plane. We introduce:
\[
\Omega_{\sigma}^2 := \left\{ s \in \C  \left| 0<\sigma< \abs{s} +\Re(s), \,\, \sigma > \frac{- \Re(s) \abs{s}}{2(\abs{s} + \Re(s))} \right.\right\},
\]
and
\[
\Omega_\sigma^3 := \left\{ s \in \C\left| \sigma(\abs{s} - \sigma +\Re(s)) - \left(\sigma \left( 1+ \frac{\Re(s)}{\abs{s}}  \right)+ \frac{1}{2}\Re(s) \right)^2>0\right.\right\},
\]
and finally:
\[
\Omega_\sigma := \Omega^1_{\sigma} \cap \left( \Omega^2_\sigma \cup \Omega^3_\sigma \right)
\]
We note the following:
\begin{Lemma}
There exists $\varphi_0 > \frac{2 \pi}{3}$ such that for all $s \in \C$ with $\abs{\arg{s}}<\varphi_0$ we can find $\sigma>0$ with $s \in \Omega_{\sigma}$.
\end{Lemma}
\begin{proof}
The boundary of $\Omega_{\sigma}$ is a piecewise smooth curve (see Figure \ref{OmFig}). From the sketch, and noting that $\Omega_{\lambda \sigma} = \lambda \Omega_{\sigma}$ for $\lambda>0$, we see that if we take $\varphi_0\in [0, \pi)$ to be the largest angle such that the line $\arg{s} = \varpi_0$ is tangent to the boundary the result will follow. At this point, the relevant condition is that determined by $\Omega^3_\sigma$ and we may verify that $\varphi_0$ is a root of:
\[
\sin^4 \varphi - \cos^2\varphi(2 + 2\cos \varphi(2 + \cos\varphi)) = 0.
\]
Numerically, we find $\varphi_0 \simeq 0.704 \pi$. 
\end{proof}

\begin{Theorem}\label{ClosedEst}
Suppose $\kappa \geq 0$, $\Upsilon \cc \Omega_\sigma$ and that $u \in C^\infty(\overline{I})$ satisfies \eq{shiftedeq}.  Suppose further that:
\[
\snorm{f}{\sigma, 0, 0}^N+\snorm{f}{\sigma, 1, 0}^N+ \ssnorm{f}{\sigma}^N <\infty,
\]
and
\[
 \snorm{\p_x u}{\sigma, 0, 0}^N+\snorm{\p_x u}{\sigma, 1, 0}^N+\snorm{\p_x u}{\sigma, 2, 0}^N + \ssnorm{\p_x u}{\sigma}^N <\infty.
\]
Then there exist constants $N_0$,  $C$, depending on $\sigma, \Upsilon$ but not on $\kappa$, such that for $N\geq N_0$:
\begin{align*}
 \snorm{\p_x u}{\sigma, 0, 0}^N+\snorm{\p_x u}{\sigma, 1, 0}^N+\snorm{\p_x u}{\sigma, 2, 0}^N + \ssnorm{\p_x u}{\sigma}^N \leq C\left( \snorm{f}{\sigma, 0, 0}^N+\snorm{f}{\sigma, 1, 0}^N+ \ssnorm{f}{\sigma}^N \right).
\end{align*}
\end{Theorem}
\begin{proof}
We add $\left( 1+ \frac{2}{N}\right)$ times the estimate of Theorem \ref{MorThm}  to the estimate of Theorem  \ref{SecEst} to obtain:
\begin{align*}
&\left[ \abs{s} - \sigma + \left( 1+\frac{2}{N} \right)\left(\Re(s) - {\textstyle \frac{\sigma}{2(N+1)}} \right) \right] \left(\snorm{\p_x u}{\sigma, 0, 0}^N \right)^2 \\
&\qquad +  \left[ 2\sigma \left( 1+ \frac{\Re(s)}{\abs{s}}  +\frac{2}{N} +  \frac{\Re(s)}{\abs{s}N }  \right) + \Re(s) \left(1+\frac{1}{N} \right)  \right] \left(\snorm{\p_x u}{\sigma, 1, 0}^N \right)^2 \\
&\qquad + \sigma \left( 1 +\frac{2}{N}  \right)\left(\snorm{\p_x u}{\sigma, 2, 0}^N\right)^2 \\&\qquad \leq C_{N} \left[ \snorm{\p_x u}{\sigma, 0, 0}^N \snorm{f}{\sigma, 0, 0}^N+  \left( \snorm{\p_x u}{\sigma, 0, 0}^{N} + \snorm{\p_x u}{\sigma, 1, 0}^{N}\right) \left( \snorm{f}{\sigma, 0, 0}^{N} + \snorm{f}{\sigma, 1, 0}^{N}\right) +\left( \ssnorm{f}{\sigma}^{N} \right)^2\right] 
\end{align*}
Let us consider the left hand side, dropping terms which are $O(N^{-1})$:
\begin{align*}
A:=\left[ \abs{s} - \sigma +\Re(s) \right] \left(\snorm{\p_x u}{\sigma, 0, 0}^N \right)^2 +  \left[ 2\sigma \left( 1+ \frac{\Re(s)}{\abs{s}}  \right)+ \Re(s) \right] \left(\snorm{\p_x u}{\sigma, 1, 0}^N \right)^2 + \sigma  \left(\snorm{\p_x u}{\sigma, 2, 0}^N\right)^2
\end{align*}
Provided that $A$ bounds $\scriptstyle \left(\snorm{\p_x u}{\sigma, 0, 0}^N \right)^2+\left(\snorm{\p_x u}{\sigma, 1, 0}^N \right)^2+\left(\snorm{\p_x u}{\sigma, 2, 0}^N \right)^2$ from above, we can close our estimate for $N$ sufficiently large. The obvious way to arrange this is to require all coefficients to be positive, so that:
\[
0<\sigma <\abs{s} + \Re(s), \quad and \quad \sigma >\frac{ - \Re(s) \abs{s}}{2(\abs{s} + \Re(s))}.
\]
Thus if $s\in \Omega^2_\sigma \cap \Omega^1_{\sigma}$, by taking $N$ large enough, we can close the estimate.  In fact, since we know that $\scriptstyle \left(\snorm{\p_x u}{\sigma, 1, 0}^N \right)^2\leq \snorm{\p_x u}{\sigma, 0, 0}^N\snorm{\p_x u}{\sigma, 2, 0}^N$, we can also permit the coefficient of $\scriptstyle \left(\snorm{\p_x u}{\sigma, 1, 0}^N \right)^2$ to be slightly negative, provided the other coefficients are sufficiently positive. We have:
\[
A \geq a \left(\snorm{\p_x u}{\sigma, 0, 0}^N \right)^2 - 2 b \left(\snorm{\p_x u}{\sigma, 0, 0}^N \right)\left(\snorm{\p_x u}{\sigma, 2, 0}^N \right)+ c \left(\snorm{\p_x u}{\sigma, 2, 0}^N \right)^2
\]
with:
\[
a=  \abs{s} - \sigma +\Re(s), \qquad b = \abs{\sigma \left( 1+ \frac{\Re(s)}{\abs{s}}  \right)+ \frac{1}{2}\Re(s)},\quad  c = \sigma.
\]
The quadratic form $a x^2 - 2b xy + c y^2$ is positive definite provided $a+c>0$ and $ac-b^2>0$. The first condition is trivially satisfied and the second gives:
\[
\sigma(\abs{s} - \sigma +\Re(s)) - \left(\sigma \left( 1+ \frac{\Re(s)}{\abs{s}}  \right)+ \frac{1}{2}\Re(s) \right)^2>0.
\]
Thus if $s\in \Omega^3_\sigma \cap \Omega^1_{\sigma}$, we can also close the estimate. Provided we restrict $s$ to a compact set  $\Upsilon \subset \mathcal{D}^\sigma$, we can close the estimate with a uniform constant and with $N$ chosen sufficiently large uniformly on $\Upsilon$.
\end{proof}
Note that if $\kappa >0$, by Theorem \ref{MorThm}, we know that 
\[
 \snorm{\p_x u}{\sigma, 0, 0}^N+\snorm{\p_x u}{\sigma, 1, 0}^N+\snorm{\p_x u}{\sigma, 2, 0}^N + \ssnorm{\p_x u}{\sigma}^N <\infty,
\]
so this condition can be dropped unless $\kappa =0$.
\begin{Corollary}
Under the same hypotheses as the previous theorem, solutions of \eq{shiftedeq} satisfying:
\[
 \snorm{\p_x u}{\sigma, 0, 0}^N+\snorm{\p_x u}{\sigma, 1, 0}^N+\snorm{\p_x u}{\sigma, 2, 0}^N + \ssnorm{\p_x u}{\sigma}^N <\infty
\]
are unique up to a polynomial of degree $N$. That is, if $u_1, u_2$ are two solutions, then $u_1-u_2$ is a polynomial of degree $N$.
\end{Corollary}

Theorem \ref{ClosedEst} permits us to estimate seminorms of $u$ which control all derivatives of order greater than $N$, uniformly in $\kappa$. 
%
%
In order to control the full Gevrey norm, we need to estimate the $H^N(I)$ norm of $u$. To do this, we recall that for a smooth function $g:I \to \C$ we have:
\[
\int_0^1 \abs{g}^2 dx \lesssim \abs{g(1)}^2 + \int_0^1 \abs{\p_xg}^2 dx.
\]

As a result, if we can estimate $\abs{u^{(n)}(1)}$ in terms of $f$ and its derivatives for $0\leq n \leq N$, then we can control $\norm{u}{\sigma, k, l}$. To do this, we return to the equation evaluated at $x=1$, and for notational convenience write $a_n := u^{(n)} (1)$ and $b_n := f^{(n)}(1)$. Differentiating \eq{shiftedeq} and setting $x=1$ yields:
\[
(\kappa +1) a_{n+2} + (\kappa (n+1) + 2 (n+1) + s+\lambda) a_{n+1} + (n(n+1) - N(N+1)) a_n = b_n.
\]
This is a second order difference equation, we wish to estimate $a_n$ with $0\leq n \leq N$ in terms of $b_n$. Our Dirichlet boundary condition at $x=1$ implies that $a_0 = 0$ and we have already estimated $a_{N+1}$, so letting $n$ vary between $0$ and $N-1$ we have a linear system of  $N$ equations for the $N$ unknowns $a_1, \ldots, a_N$. In fact, we may write this as the matrix equation:
\[
(A_N+\kappa B_N +(s+\lambda)I_N)w = v,
\]
where $A_N$ is the tridiagonal matrix:
\[
A_N= \left(\begin{array}{ccccc}2 & 1 & 0 & \cdots & 0 \\2-N(N+1) & 4 & 1 &   & \vdots \\0 & 6-N(N+1) & 6 &  & 0 \\\vdots &   &   & \ddots & 1 \\0 & \cdots &   & -2N & 2N\end{array}\right),
\]
$B_N$ is the banded upper-diagonal matrix:
\[
B_N =\left(\begin{array}{ccccc}1 & 1 & 0 & \cdots & 0    \\0 & 2 & 1 &   & \vdots    \\ 0&0&3 && \\   \vdots &   & & \ddots    & 0    \\  &   &   &  & 1    \\0 & \cdots &   & 0 & N   \end{array}\right),
\]
and the vectors $w$, $v$ are formed from $a_n, b_n$ as:
\be
w = \left(\begin{array}{c}a_1 \\\vdots \\a_{N}\end{array}\right),\qquad v=\left(\begin{array}{c}b_0 \\\vdots \\b_{N-2} \\ b_{N-1} - (\kappa+1)a_{N+1}\end{array}\right).
\ee
Clearly, if we fix $\Upsilon \cc \Omega_\sigma$ and $\kappa_0>0$, then provided we take $\lambda = \lambda(\Upsilon, \kappa_0, N)$ sufficiently large, we have a bound:
\[
\norm{w}{} \leq C \norm{v}{}
\]
for a constant $C$ which depends on $\Upsilon$, $\kappa_0$, but not on $s \in \Upsilon$ or $\kappa$ with $0\leq \kappa<\kappa_0$. As a consequence, we can show:
\begin{Theorem}\label{NormEst}
Suppose $\Upsilon \cc \Omega_{\sigma}$, fix $\kappa_0>0$, let $N_0$ be as in Theorem \ref{ClosedEst}, and suppose $N\geq N_0$. Let $\lambda = \lambda(\Upsilon, \kappa_0, N)$ be as above. Assume $u \in X^\sigma$ satisfies \eq{shiftedeq}. Then there exists a constant $C$ depending on $N, \Upsilon, \kappa_0$ but not $\kappa$ such that for any $0\leq \kappa<\kappa_0$ we have:
\[
 \norm{\p_x u}{\sigma, 0, 0}+\norm{\p_x u}{\sigma, 1, 0}+\norm{\p_x u}{\sigma, 2, 0} + \ssnorm{\p_x u}{\sigma}^0   \leq C\left( \norm{f}{\sigma, 0, 0}+\norm{f}{\sigma, 1, 0}+ \ssnorm{f}{\sigma}^0 \right).
\]
\end{Theorem}
\begin{proof}
From Theorem \ref{ClosedEst} we have:
\[
 \snorm{\p_x u}{\sigma, 0, 0}^N+\snorm{\p_x u}{\sigma, 1, 0}^N+\snorm{\p_x u}{\sigma, 2, 0}^N + \ssnorm{\p_x u}{\sigma}^N \leq C \left( \norm{f}{\sigma, 0, 0}+\norm{f}{\sigma, 1, 0}+ \ssnorm{f}{\sigma}^0 \right).
\]
In particular, this gives control of $a_{n+1}$ in the notation above by a trace estimate. Thus in particular,
\[
\norm{v}{} \leq C \left( \norm{f}{\sigma, 0, 0}+\norm{f}{\sigma, 1, 0}+ \ssnorm{f}{\sigma}^0 \right).
\]
Finally, we have that:
\[
\norm{w}{} \leq C \norm{v}{}
\]
As a consequence we deduce that:
\[
\norm{w}{} \leq C \left( \norm{f}{\sigma, 0, 0}+\norm{f}{\sigma, 1, 0}+ \ssnorm{f}{\sigma}^0 \right).
\]
With the observations above, this completes the proof.
\end{proof}
With this result, we're able to show that  we can solve \eq{shiftedeq} uniquely for $s \in \Omega_\sigma$, working in the appropriate Gevrey classes.
\begin{Theorem}\label{L0Thm}
Suppose $\Omega \cc \Omega_{\sigma}$, let $N_0$,  be as in Theorem \ref{ClosedEst}, and fix $N\geq N_0$ and $\lambda$ as in Theorem \ref{NormEst}. Then for any $s \in \Omega$, given $f \in Y^\sigma$ there exists a unique solution $u \in X^\sigma$ to \eq{shiftedeq} with $u(1)=0$. Setting $u = R(s) f$ for some operator $R(s) : Y^\sigma \to X^\sigma$, we have that $R(s)$ is holomorphic on $\Omega$.
\end{Theorem}
\begin{proof}
By continuity of the bounds defining $\mathcal{D}^\sigma$, we can find $\sigma'>\sigma$ such that $\Omega\subset D_{\sigma'}$. Fix a sequence $0<\kappa_i<\kappa_0$ with $\kappa_i \to 0$. We know that problem \eq{shiftedeq} is Fredholm for $\kappa >0$. The estimate of Theorem \ref{NormEst} shows that solutions to  \eq{shiftedeq} are unique, hence for each $\kappa_i$ we have $u_i \in X^{\sigma'}$ solving:
\[
\frac{d}{dx}\left( (\kappa_i x + x^2) \frac{d u_i}{dx} \right) + s \frac{du_i}{dx} -N(N+1) u_i+ \lambda \sum_{j=0}^{N-1} \frac{(x-1)^j}{j!} u_i^{(j+1)}(1) = f
\]
and moreover $\norm{u_i}{X^{\sigma}}$ and $\norm{u}{\sigma', 0, 0}$ are uniformly bounded. Hence, we can extract a subsequence which converges strongly in $G_{\sigma,0,0}$ to $u \in X^\sigma$ as $\kappa \to 0$, which in particular implies:
\[
\frac{d}{dx}\left( x^2\frac{d u}{dx} \right) + s \frac{du}{dx} -N(N+1) u+ \lambda \sum_{j=0}^{N-1} \frac{(x-1)^j}{j!} u^{(j+1)}(1)= f .
\]
Uniqueness follows immediately from Theorem \ref{NormEst} with $\kappa = 0$. To see that $R(s)$ is holomorphic, we set $u_s := R(s) f$, so that:
\[
\frac{d}{dx}\left( x^2\frac{d u_s}{dx} \right) + s \frac{du_s}{dx} -N(N+1) u_s + \lambda \sum_{j=0}^{N-1} \frac{(x-1)^j}{j!} u_s^{(j+1)}(1)= f
\]
A simple calculation shows that $w = u_s - u_{s'}$ satisfies:
\[
\frac{d}{dx}\left( x^2\frac{d w}{dx} \right) + s \frac{dw}{dx} -N(N+1) w+ \lambda \sum_{j=0}^{N-1} \frac{(x-1)^j}{j!} w^{(j+1)}(1)= -(s-s') \frac{du_{s'}}{dx}
\]
By our estimates above, we have that:
\[
\norm{w}{X^{\sigma}} \leq C \abs{s-s'} \norm{f}{Y^{\sigma}}
\]
so that $R(s):Y^\sigma \to X^\sigma$ is continuous in the operator norm topology as $s$ varies. Moreover 
\[
\frac{R(s) - R(s')}{s-s'} = -R(s) \circ \frac{d}{dx} \circ R(s') \to -R(s) \circ \frac{d}{dx} \circ R(s)
\]
as $s' \to s$, where the convergence is again in operator norm and we use that $\frac{d}{dx} : X^\sigma \to Y^\sigma$ is bounded.
\end{proof}
This completes the proof of Proposition \ref{mainprop}. We include here some further useful results.
\begin{Lemma}
Fix $\sigma >0$. Then for $s \in \Omega_\sigma$, the domain of the operator $\mathcal{D}(\mathcal{L}_s)$ (defined as the set of $u \in X^\sigma$ such that $\mathcal{L}_s u \in Y^\sigma$) depends on $\sigma$, but not on $s$. That is to say $\mathcal{D}(\mathcal{L}_s)= \mathcal{D}(\mathcal{L}_{s'})=:\mathcal{D}^\sigma$ for $s, s' \in \Omega_\sigma$.
\end{Lemma}
\begin{proof}
By Theorem \ref{DThm}  we know that if $u \in X^\sigma$, then $(\mathcal{L}_s-\mathcal{L}_{s'})u \in Y^\sigma$ and hence $\mathcal{L}_s u \in Y^\sigma \iff \mathcal{L}_{s' }u \in Y^\sigma$.
\end{proof}
Next, we show that the construction of quasinormal modes above is independent of the choice of $\sigma$, i.e. the location of the poles of $\mathcal{L}_s^{-1} : Y^\sigma \to Y^\sigma$ does not depend on the choice of $\sigma$. In order to keep track of $\sigma$, we define $\textrm{Ker}_\sigma \mathcal{L}_s := \{ u \in \mathcal{D}^\sigma | \mathcal{L}_s u = 0\}$; $\textrm{Ran}_\sigma\mathcal{L}_s:= \mathcal{L}_s \mathcal{D}^\sigma \subset Y^\sigma$; $\textrm{Coker}_\sigma\mathcal{L}_s := Y^\sigma / \textrm{Ran}_\sigma \mathcal{L}_s$. 

\begin{Theorem}
Let $0<\sigma<\sigma'$ and suppose $s \in \Omega_\sigma \cap \Omega_{\sigma'}$. Then 
\[
\textrm{dim }\textrm{Ker}_\sigma \mathcal{L}_s = \textrm{dim }\textrm{Ker}_{\sigma'} \mathcal{L}_s.
\]
\end{Theorem}
\begin{proof}
Since by Proposition \ref{mainprop} the map $\mathcal{L}_s:\mathcal{D}^\sigma \to Y^\sigma$ is Fredholm, we know that $\textrm{dim }\textrm{Ker}_\sigma \mathcal{L}_s=\textrm{dim }\textrm{Coker}_\sigma \mathcal{L}_s<\infty$, and that $\textrm{Ran}_\sigma \mathcal{L}_s$ is closed in $Y^\sigma$. In particular by the closed graph theorem we can write $\textrm{Ran}_\sigma \mathcal{L}_s = \{ u \in Y^\sigma | \omega_i(u) = 0, i=1, \ldots, k\}$, where $\omega_1, \ldots \omega_k$ is a linearly independent subset of $(Y^\sigma)^*$ and $k =\textrm{dim }\textrm{Coker}_\sigma \mathcal{L}_s$. As a result, we can identify $\textrm{Coker}_\sigma \mathcal{L}_s\cong \langle \omega_1, \ldots, \omega_k \rangle$.

Now, suppose $u \in \textrm{Ker}_{\sigma'} \mathcal{L}_s$. Since $\mathcal{D}^{\sigma'} \subset \mathcal{D}^\sigma$, we immediately have that $u \in \textrm{Ker}_{\sigma} \mathcal{L}_s$ and so:
\[
\textrm{dim }\textrm{Ker}_\sigma \mathcal{L}_s \geq \textrm{dim }\textrm{Ker}_{\sigma'} \mathcal{L}_s.
\]
Next suppose that $\omega \in \textrm{Coker}_\sigma \mathcal{L}_s$, where as above we have identified the cokernel with a subspace of $(Y^\sigma)^*$. Then $\omega(u) = 0$ for all $u \in \textrm{Ran}_\sigma\mathcal{L}_s$. However, $\textrm{Ran}_{\sigma'}\mathcal{L}_s \subset \textrm{Ran}_\sigma\mathcal{L}_s$ since $\mathcal{D}^{\sigma'} \subset D^{\sigma}$. Thus  $\omega(u) = 0$ for all $u \in \textrm{Ran}_{\sigma'}\mathcal{L}_s$, hence $\omega \in \textrm{Coker}_{\sigma'} \mathcal{L}_s$. This implies:
\[
\textrm{dim }\textrm{Coker}_\sigma \mathcal{L}_s \leq \textrm{dim }\textrm{Coker}_{\sigma'} \mathcal{L}_s.
\]
Using the Fredholm property, we are done.
\end{proof}

Now we establish an elliptic regularity result away from $x=0$.
\begin{Theorem}\label{IntReg}
Suppose $u \in H^2(I)$  has support in $[\delta, 1]$ for some $\delta>0$, vanishes at $x=1$ in the trace sense and satisfies $\mathcal{L}_s u = f$ for some $f \in Y^\sigma$. Then in fact $u \in X^\sigma$ and 
\[
\norm{u}{X^\sigma} \leq C\left(  \norm{f}{Y^\sigma} + \norm{u}{L^2(I)} \right)
\]
for some constant depending on $\sigma, \delta, W$. 
\end{Theorem}
\begin{proof}
By a standard elliptic estimate, we know that $u \in C^\infty(\overline{I})$, $u(1)=0$. Fix $0<N'\leq M$. We return to the result of Lemma \ref{GMest} with $\kappa = 0$ and $N = 0$ to deduce:
\begin{align*}
&\left(2 \sigma+\Re(s) \right) \left( \snorm{\p_x u}{\sigma, 1, 0}^{N', M}\right)^2 + 2\sigma \left( \snorm{\p_x u}{\sigma, 2, 0}^{N', M}\right)^2  \\
&\qquad \leq \left( \textstyle\frac{\sigma}{2(N'+1)} -\Re(s)\right)\left(\snorm{\p_x u}{\sigma, 0, 0}^{N', M} \right)^2  \\
&\quad \qquad + \left( \snorm{\p_x u}{\sigma, 0, 0}^{N', M} + \snorm{\p_x u}{\sigma, 1, 0}^{N', M}\right) \left( \snorm{f-Wu}{\sigma, 0, 0}^{N', M} + \snorm{f-Wu}{\sigma, 1, 0}^{N', M}\right) \\
&\qquad \qquad + {\textstyle\frac{\sigma^{2N'}}{N'!^2 (N'+1)!^2} \frac{N'}{2 \sigma}  N'(N'+1)}\int_0^1 \abs{u^{(N')}}^2 dx 
\end{align*}
From here, making use of the remark under Theorem \ref{ProdThm} we deduce:
\[
\left( \snorm{\p_x u}{\sigma, 2, 0}^{N', M}\right)^2 \leq c \left(\left( \snorm{\p_x u}{\sigma, 0, 0}^{N', M}\right)^2+ \left( \snorm{\p_x u}{\sigma, 1, 0}^{N', M}\right)^2 + \left( \snorm{f}{\sigma, 1, 0}^{N', M}\right)^2\right) + C(N') \norm{u}{H^{N'}(I)}^2
\]
for some $c$ independent of $N'$. Now, since $\textrm{supp }u \subset [\delta, 1]$, we have 
\[
\left( \snorm{\p_x u}{\sigma, 2, 0}^{N', M}\right)^2 \geq N' \frac{\delta}{\sigma} \left( \snorm{\p_x u}{\sigma, 1, 0}^{N', M}\right)^2\geq N'^2 \frac{\delta^2}{\sigma}^2 \left( \snorm{\p_x u}{\sigma, 0, 0}^{N', M}\right)^2
\]
As a consequence, for $N'$ taken sufficiently large we deduce
\[
\left( \snorm{\p_x u}{\sigma, 0, 0}^{N', M}\right)^2+\left( \snorm{\p_x u}{\sigma, 1, 0}^{N', M}\right)^2+ \left( \snorm{\p_x u}{\sigma, 2, 0}^{N', M}\right)^2  \leq c \left( \snorm{f}{\sigma, 1, 0}^{N', M}\right)^2 + C(N') \norm{u}{H^{N'}(I)}^2
\]
A standard elliptic estimate gives
\[
\norm{u}{H^{N'+1}(I)}^2 \leq C(N') \norm{f}{H^{N'-1}(I)}^2 +  \norm{u}{L^2(I)}^2
\]
so that:
\[
\left( \snorm{\p_x u}{\sigma, 0, 0}^{0, M}\right)^2+\left( \snorm{\p_x u}{\sigma, 1, 0}^{0, M}\right)^2+ \left( \snorm{\p_x u}{\sigma, 2, 0}^{0, M}\right)^2  \leq c\left(  \left( \snorm{f}{\sigma, 1, 0}^{0, M}\right)^2 +\norm{u}{L^2(I)}^2\right)
\]
where the constant $c$ is independent of $M$. Sending $M$ to infinity we are done.
\end{proof}
Note that in the result above we do not require $s$ to be a point at which $\mathcal{L}_s$ is invertible. If we are at such a point then we can drop the $L^2$-norm of $u$ from the right-hand side of the estimate. 

Finally, we show that when $f\in Y^\sigma$ is supported away from $x=0$ then we can control the $H^2$-norm of the solution to $\mathcal{L}_s u = f$ by the $L^2$-norm of $f$, which will permit us to relax our regularity conditions for data supported away from $x=0$.
\begin{Theorem}
Suppose that $\mathcal{L}_s:\mathcal{D}^\sigma \to Y^\sigma$ is invertible at $s_0 \in \Omega_\sigma$ and fix $\delta>0$. Then there exists a neighbourhood $U \subset \Omega_\sigma$ of $s_0$ and a constant $C$ depending on $U, \delta, W$ such  for any $f \in Y^\sigma$ with support in $[\delta, 1]$ for $\delta >0$ we have:
\ben{MollEst}
\norm{\mathcal{L}_s^{-1} f}{H^2(I)} \leq C\norm{f}{L^2(I)}
\een
for all $s \in U$.
\end{Theorem}
\begin{proof}
Pick $U\cc \Omega_\sigma$ a neighbourhood of $s_0$ such that $\mathcal{L}_s:\mathcal{D}^\sigma \to Y^\sigma$ is invertible for $s \in U$. Setting $u = \mathcal{L}^{-1}_sf$, we know that $u$ satisfies
\[
\frac{d}{dx} \left(x^2 \frac{du}{dx} \right) + s \frac{du}{dx} - W u = f
\]
with $u(1) = 0$. Let $w(x) = u(\delta x) - u(\delta)$. This satisfies
\[
\tilde{\mathcal{L}}_{s\delta^{-1}}w:=\frac{d}{dx} \left(x^2 \frac{dw}{dx} \right) + \frac{s}{\delta} \frac{dw}{dx} - \tilde{W} w= u(\delta) W
\]
and $w(1) = 0$, where $\tilde{W}(x) = W(\delta x)$. Noting that this operator is invertible except at isolated values of $s\delta^{-1}$, by making $\delta$ slightly smaller if necessary, and shrinking $U$ we can assume $\tilde{\mathcal{L}}_{s\delta^{-1}}:\mathcal{D}^\sigma \to Y^\sigma$ is invertible for $s \in U$ and moreover $w \in X^{\sigma\delta^{-1}}$ with:
\[
\norm{\p_x w}{\sigma\delta^{-1}, 0, 0}+\snorm{\p_x w}{\sigma\delta^{-1}, 1, 0}^0+\snorm{\p_x w}{\sigma\delta^{-1}, 2, 0}^0 + \ssnorm{\p_x w}{\sigma\delta^{-1}}^0 \leq C\abs{u(\delta)} \leq C \norm{\p u}{L^2(I)}^\frac{1}{2} \norm{ u}{L^2(I)}^\frac{1}{2},
\]
where we have used a standard trace estimate in the second inequality, and the constants may depend on $U, \delta, W$ but not on $w$. Changing variables back, we deduce:
\ben{ACEst1}
\norm{\p_x u|_{[0,\delta]}}{\sigma, 0, 0}+\snorm{\p_x u|_{[0,\delta]}}{\sigma, 1, 0}^0+\snorm{\p_x u|_{[0,\delta]}}{\sigma, 2, 0}^0 + \ssnorm{\p_x u}{\sigma}^0 \leq C \norm{\p u}{L^2(I)}^\frac{1}{2} \norm{ u}{L^2(I)}^\frac{1}{2}.
\een
By the notation $\norm{ u|_{[0,\delta]}}{\sigma, k, l}$ we mean the $\norm{u}{\sigma, k, l}$ seminorm defined as above, but with the range of integration restricted to $[0, \delta]$. In particular this implies
\[
\norm{u}{H^2([0,\delta])} \leq C \norm{u}{L^2(I)}
\]
Standard elliptic estimates give
\[
\norm{u}{H^2([\delta,1])} \leq C\left( \norm{f}{L^2(I)} + \norm{u}{L^2(I)}\right)
\]
so that we have the global estimate
\[
\norm{u}{H^2(I)} \leq C \left( \norm{f}{L^2(I)} + \norm{u}{L^2(I)}\right)
\]
which holds for all $f \in Y^\sigma$ with support in $[\delta, 1]$, and where the constant depends on $U, \delta, W$.

Next we show that if $V \cc U$, then there exists a constant $C$ such that for any $s \in W$ and $f \in Y^\sigma$ with support in $[\delta, 1]$ we have:
\[
\norm{u}{H^2(I)} \leq C_{\delta}  \norm{f}{L^2(I)}.
\]
Suppose not, then there exists a sequence of $s_k \in V$ and $f_k \in Y^\sigma$ with support in $[\delta, 1]$ such that, letting $u_k = \mathcal{L}^{-1}_{s_k} f_k$ we have $\norm{u_k}{L^2(I)} = 1$ and $\norm{f_k}{L^2(I)} \to 0$. By Bolzano--Weierstrass we can assume (after taking a subsequence) that $s_k \to s\in U$. By the global estimate above we have that $(u_k)$ is bounded in $H^2(I)$, so after extracting a subsequence, we can assume $u_k$ converges weakly in $H^2(I)$ and strongly in $H^1(I)$ to some $u$ which solves $\mathcal{L}_s u = 0$ and satisfies $u(1)=0$, $\norm{u}{L^2(I)} = 1$. By estimate \eq{ACEst1}, we know that (again up to a subsequence) the limiting function $u$ satisfies 
\[
\norm{\p_x u|_{[0,\delta]}}{\sigma, 0, 0}+\snorm{\p_x u|_{[0,\delta]}}{\sigma, 1, 0}^0+\snorm{\p_x u|_{[0,\delta]}}{\sigma, 2, 0}^0 + \ssnorm{\p_x u}{\sigma}^0 <\infty.
\]
Moreover, by Theorem \ref{IntReg} applied to the function $\chi u$, where $\chi$ is a $(\sigma', 2)-$Gevrey regular cut-off function with $\chi(x) = 1$ for $x>\delta$ and $\chi(x) = 0$ for $x<\delta/2$, we know that $\chi u\in X^\sigma$. By the invertibility of $\mathcal{L}_s : X^\sigma \to Y^\sigma$, we conclude that $u=0$, contradicting $\norm{u}{L^2(I)} = 1$. We conclude that $\norm{u}{H^2(I)} \leq C_{\delta}  \norm{f}{L^2(I)}$ and the result follows on combining this with the global estimate.
\end{proof}
This immediately gives us
\begin{Corollary}\label{MECor}
Fix $\delta >0$ and $\sigma'>\sigma>0$. Suppose $\chi\in C^\infty(I)$ is a $(\sigma', 2)-$Gevrey regular cut-off function with $\chi(x) = 1$ for $x>2\delta$ and $\chi(x) = 0$ for $x<\delta$. Then $\mathcal{L}_s^{-1} \circ \chi:Y^\sigma \to \mathcal{D}^\sigma$ extends by continuity to a meromorphic family of operators, denoted with the same letter by an abuse of notation, $A_s: L^2(I) \to H^2(I)$ for $s \in \Omega_\sigma$. Poles of $A_s$ may occur only where $\mathcal{L}_s^{-1}$ has poles, however the degree  of each pole of $A_s$ may be less than the corresponding pole of $\mathcal{L}_s^{-1}$.
\end{Corollary}
\begin{proof}
Given $f \in L^2(I)$, we can find an approximating sequence $f_k \in Y^\sigma$ with $f_k \to f$ in $L^2(I)$ by mollifying with a $(\sigma', 2)-$Gevrey regular bump function. The previous theorem tells us that if $s$ is not a pole of $\mathcal{L}_s^{-1}$, then $\mathcal{L}_s^{-1}( \chi f_k)$ converges to a limit $A_sf\in H^2(I)$ as $k \to \infty$, locally uniformly in $s$. By Morrera's theorem we have that $A_s:L^2(I) \to H^2(I)$ is holomorphic away from poles of  $\mathcal{L}_s^{-1}$. If $s_0$ is a pole of degree $d$, then by considering $(s-s_0)^d \mathcal{L}_s^{-1}( \chi f_k)$ and again applying Morrera's theorem we have that  $(s-s_0)^dA_s:L^2(I) \to H^2(I)$ is holomorphic at $s=s_0$, and hence $A_s:L^2(I) \to H^2(I)$ has a pole of order at most $d$.
\end{proof}

\section{Connection to other definitions}

\subsection{Meromorphicity of the resolvent}
In this section we establish the results of Theorem \ref{Thm1} in the introduction. In particular, we relate the eigenvalue problem for the null $(t, x)$ coordinates discussed above with the meromorphic extension of the resolvent for the original $(\tau,r)$ coordinates. We recall the Laplace transformed operator in the original $(\tau, r)$ variables is given by:
\[
\hat{L}_sw := \frac{d^2w}{dr^2} - \left(\frac{s^2}{4} + V(r)\right)w.
\]
It follows from standard semi-group theory (or directly from the equation) that $\hat{L}_s:H^2(\R_{>1})\cap H^1_0(\R_{>1}) \to L^2(\R_{>1})$ is invertible, and moreover that the inverse $\hat{L}_s$ is holomorphic in $\Re(s)>0$. We wish to show that $\hat{L}_s^{-1}: L^2_{c}(\R_{>1}) \to H^2_{loc.}(\R_{>1})$ admits a meromorphic extension to a domain which enters the left half-plane. To do this, we relate $\hat{L}_s$ to $\mathcal{L}_s$.

\begin{Lemma}
Let $w, g: \R_{\geq 1} \to \C$. Define new functions $u := P_sw : (0, 1] \to \C$ and $f :=Q_s g:(0,1] \to \C$ by
\[
u(x) = e^{\frac{s}{2x}} w\left(\frac{1}{x} \right), \qquad f(x) = \frac{1}{x^2} e^{\frac{s}{2x}} g \left(\frac{1}{x} \right).
\]
Then:
\[
\hat{L}_s w = g
\]
if and only if
\[
\mathcal{L}_s u = f.
\]
Equivalently, we have that
\[
\hat{L}_s = Q_s^{-1} \circ \mathcal{L}_s \circ P_s.
\]
\end{Lemma}\label{cov}
\begin{proof}
This is a straightforward computation.
\end{proof}
The crucial observation concerning the maps $P_s, Q_s$ is that they are well behaved away from $x=0$ (equivalently $r=\infty$). It is this singular behaviour which forces us to consider modified spaces when we wish to meromorphically continue $\hat{L}_s^{-1}$. We are now able to complete the proof of Theorem \ref{Thm1}.

\begin{proof}[Proof of Theorem \ref{Thm1}]
Suppose $V$ is a Type $II_{\sigma'}$ potential. For any $R>1$, we can pick $\chi_R \in C^\infty_c(\R_{\geq 1})$ such that $\chi_R(r) = 1$ for $x<R$ and $\chi(R)=0$ for $x>2R$. We define:
\[
A_{s, R} := \chi_R \circ P_s^{-1} \circ \mathcal{L}_s^{-1} \circ Q_s \circ \chi_R.
\]
We claim that for any $\sigma <\sigma'$, $A_{s, R} : L^2(\R_{>1}) \to H^2(\R_{>1})$ is well defined for $s \in \Omega_\sigma$, away from poles of $\mathcal{L}_s^{-1}$. To see this, we note that for $g \in  L^2(\R_{>0})$ we have that $\chi_R g$ is supported away from $r = \infty$, and it is easy to see that $Q_s (\chi_Rg) :(0, 1] \to \C$ belongs to $L^2(I)$, depends meromorphically on $s$ and moreover vanishes near $x=0$, since the map $Q_s$ is regular away from $r=\infty$. As a result, by Corollary \ref{MECor} $\mathcal{L}_s^{-1}(Q_s (\chi_Rg)) \in H^2(I)$ has a meromorphic extension to $\Omega_{\sigma}$. Finally, the map $\chi{R}\circ  P_s^{-1}$ maps $H^2(I)$ holomorphically into $H^2(\R_{>1})$ since the cut-off removes any issues with growth near $r=\infty$. We have thus shown that $\chi_R \circ \hat{L}_s^{-1} \circ \chi_R:L^2(\R_{>1}) \to H^2(\R_{>1})$ extends as a meromorphic family of operators to $\cup_{\sigma<\sigma'}\Omega_{\sigma}$. This is precisely the statement that $\hat{L}_s^{-1} :L^2_c(\R_{>1}) \to H^2_{loc.}(\R_{>1})$ extends as a meromorphic family of operators on the same domain. Finally, noting that the potentials considered in Theorem \ref{Thm1} give rise to transformed potentials $W$ which are polynomial, and hence $(\sigma', 2)-$Gevrey for all $\sigma'$, we conclude that $\hat{L}_s^{-1} :L^2_c(\R_{>1}) \to H^2_{loc.}(\R_{>1})$ extends meromorphically to $\cup_{\sigma}\Omega_{\sigma}$, which is a sector of opening angle $\theta > \frac{\pi}{2}$. Each pole corresponds to a finite dimensional space of eigenvalues of $\mathcal{L}_s$, which after applying the transformations gives a finite dimensional space of solutions to $\hat{L}_sw = 0$ which are outgoing in the sense that $w(r) = e^{-\frac{s}{2} r} u(r^{-1})$, where $u \in X^\sigma$ for each $\sigma$ such that $s \in \Omega_\sigma$.
\end{proof}
We have thus established that the scattering resonances are well defined, and moreover shown the scattering resonances to be quasinormal frequencies.

\subsection{The method of Leaver}
For practical computations, the method of Leaver \cite{Leaver, Leaver2} is often applied when seeking quasinormal frequencies. This method can be adapted to our setting as follows. We follow  the presentation of Ansorg and Macedo \cite{AnsMac}, the arguments are somewhat heuristic, but ultimately we shall arrive at a clean definition. We will assume for simplicity that $W(x)$ is a polynomial of order $p$ in $x$ (in fact it would suffice to assume $W$ is analytic with a large enough radius of convergence). Suppose $s$ is in the left half-plane. We can then develop the general solution to
\[
\mathcal{L}_s u = \frac{d}{dx}\left( x^2 \frac{du}{dx} \right) + s \frac{du}{dx} -Wu = 0
\]
satisfying the boundary condition at $x=1$ as a power series about $x=1$, i.e.
\ben{LeaverSeries}
u(x) = \sum_{k=1}^\infty H_k (1-x)^k.
\een
Standard results tell us that this series should converge on $(0,1]$. Leaver's approach is to consider the general asymptotics of the expansion coefficients $H_k =  (-1)^ku^{(k)}(1)/k!$. In view of the expected radius of convergence, we anticipate that $\abs{H_{k+1}/H_k} \to 1$ as $k \to \infty$. Inserting our power series ansatz, we find that $H_k$ should obey a recurrence relation:
\begin{align*}
0&= (k+2)(k+1) H_{k+2} - (k+1) H_{k+1} (2(k+1) + s) H_{k+1} + k(k+1)H_k \\&\qquad - \sum_{l=0}^p \frac{W^{(l)}(1)}{l!} (-1)^l H_{k-l}
\end{align*}
where we set $H_{k}=0$ for $k<1$ and $H_1=1$ to fix a scaling. We define $\rho_k := H_{k+1}/H_k$, and assuming that for sufficiently large $k$ we have $\abs{\rho_k-1}<\frac{1}{2}$ we can deduce:
\[
\left(1+\frac{2}{k} \right) \rho_{k+1} - \left(2+\frac{s+2}{k} \right)  + \frac{1}{\rho_k} = O(k^{-2})
\]
We find that this equation will be satisfied if $\rho_k$ has the asymptotic form:
\[
\rho_k = 1 \pm \sqrt{\frac{s}{k}} +\frac{\frac{s}{2} - \frac{3}{4}}{k} + O(k^{-\frac{3}{2}}).
\]
Here and below, for $\abs{\arg{z}}<\pi$ we define $\sqrt{z}$ to be the branch of the square-root satisfying $\Re{\sqrt{z}} >0$. Taking logarithms, we deduce that:
\[
\log H_{k+1} = \log H_{k}  \pm \sqrt{\frac{s}{k}} +\frac{\frac{s}{2} - \frac{3}{4}}{k} + O(k^{-\frac{3}{2}})
\]
Summing from some $k_0$, we deduce:
\[
\log H_{k+1} = \pm \sqrt{s} \sum_{j=k_0}^k\frac{1}{\sqrt{j}} - \frac{3}{4} \sum_{j=k_0}^k\frac{1}{j} + O(1) 
\]
Now, noting the asymptotic behaviour of the generalised harmonic series:
\[
\sum_{j=1}^k \frac{1}{\sqrt{j}} = 2 \sqrt{k} + O(1), \qquad \sum_{j=1}^k \frac{1}{j} = \log k + O(1)
\]
we conclude that for sufficiently large $k$ we have:
\[
H_k = k^{-\frac{3}{4}} \left(A^+_k e^{2 \sqrt{s k}} + A^-_k e^{-2 \sqrt{s k}}  \right),
\]
where $A^\pm_k$ converge to finite values as $k \to \infty$, which we denote $A^\pm_{\infty}$. Leaver argued that in order to identify the quasinormal frequencies, one should impose the condition that $A^+_\infty=0$, which implies that the series \eq{LeaverSeries} converges uniformly. Our discussion here has been somewhat heuristic, although it could in principle be made rigorous. Nevertheless it suffices for us to motivate the following definition:
\begin{Definition}
We say that $s$ with $\Re(s)<0$ is a quasinormal frequency in the sense of Leaver if there exists a solution $u$ to the homogeneous equation $\mathcal{L}_s u = 0$ which is of the form \eq{LeaverSeries} with $H_k$ satisfying
\[
\sup_{k} \abs{H_k e^{2 \sqrt{s k}}} < \infty.
\]
\end{Definition}

Leaver's method of continued fractions gives a computational approach to find $s$ such that the corresponding $H_k$ satisfy the Leaver condition. Crudely, one solves the difference equation for $H_k$ `backwards from infinity' through a continued fraction, and then imposes the condition that $H_0=0$. This gives a formal equation involving a function of $s$ defined through continued fractions whose roots one seeks. It is not a priori clear that this function is well defined on a reasonable subset of $\C$, nor that its roots are discrete. Even assuming these facts, it's unclear how the spectrum so obtained relates to the original time evolution problem.

We shall show that if $s$ is a quasinormal frequency in the sense of Leaver, then it is also a quasinormal frequency in the sense we have introduced above. In particular this shows that the Leaver QNFs are discrete, and it furthermore connects them directly to the original evolution problem. Hitherto we are not aware of any work which establishes that the Leaver definition defines discrete frequencies, nor of a justification for why these frequencies should be relevant for an evolution problem.

\begin{Theorem}\label{LeaverThm}
Fix $\sigma>0$ and suppose that $s$ is a quasinormal frequency in the sense of Leaver with $\sigma < \Re(\sqrt{s})^2$. Then the corresponding power series solution $u$ belongs to $X^\sigma$, and hence if $s \in \Omega_{\sigma}$ it is a quasinormal frequency in the sense of Proposition \ref{mainprop}.
\end{Theorem}
\begin{proof}
We are required to show that if
\[
u(x) = \sum_{k=1}^\infty H_k (1-x)^k.
\]
with $H_k$ satisfying:
\[
\abs{H_k} \leq C e^{-2 \Re(\sqrt{s}) \sqrt{k}}
\]
for all $k$ then $u \in X^\sigma$. Differentiating the power series we have:
\[
u^{(n)}(x) = \sum_{k=1}^\infty H_k \frac{k!}{(k-n)!} (1-x)^{k-n}.
\] 
We see that this converges uniformly on $(0, 1]$, and so we deduce that $u \in C^\infty(\overline{I})$ and moreover:
\[
\abs{u^{(n)}(x)} \leq \sum_{k=1}^\infty k^n \abs{H_k}.
\]

Now, let us consider the series
\[
g(\lambda) := \sum_{k=1}^\infty e^{-\lambda \sqrt{k}}
\]
This defines a holomorphic function for $\lambda \in \{\Re{z}>0\}$ and as a consequence we have the estimate:
\[
\abs{g^{(n)}(\lambda)} \leq C_{\theta} (\theta \Re(\lambda))^{-n} n!
\]
for any $\theta \in (0,1)$. Differentiating the series directly, we have that
\[
g^{(2n)}(\lambda) =  \sum_{k=1}^\infty k^n e^{-\lambda \sqrt{k}}
\]
so we consequently deduce
\[
\abs{\sum_{k=1}^\infty k^n e^{-\lambda \sqrt{k}}} \leq C (\theta \Re(\lambda))^{-2n} (2n)!
\]

Taking $\lambda = 2 \Re(\sqrt{s})$, we find:
\[
\abs{u^{(n)}(x)} \leq \sum_{k=1}^\infty k^n \abs{H_k} \leq \sum_{k=1}^\infty k^n e^{-2 \Re(\sqrt{s}) \sqrt{k}} \leq C_{\theta} (2 \theta \Re(\sqrt{s}))^{-2n} (2n)!.
\]
Recalling the definition of the $X^\sigma$ norm, we see that this estimate implies $u \in X^\sigma$, provided:
\[
\sum_{n=1}^\infty n^2 \sigma^{2n} \frac{(2 \theta \Re(\sqrt{s}))^{-4(n+1)} (2(n+1))!^2}{n!^2 (n+1)!^2} <\infty
\]
Standard estimates for the central binomial coefficient tell us that:
\[
\frac{(2(n+1))!^2}{(n+1)!^4} \leq 4^{2(n+1)}
\]
so we see that provided
\[
\sigma < \Re(\sqrt{s})^2
\]
we can find a $\theta$ such that the sum converges.
\end{proof}
\begin{Lemma}
There exists $\varphi_1>\frac{\pi}{2}$ such that for any $s$ with $\frac{\pi}{2}< \abs{\arg s}<\varphi_1$ we can find $\sigma$ with $s \in \Omega_\sigma$ and  $\sigma < \Re(\sqrt{s})^2$.
\end{Lemma}
\begin{proof}
This follows from the definition of $\Omega_\sigma$. In particular, if $\abs{\arg s}<\varphi_1$ where $\varphi_1$ is the solution of
\[
4\cos^4 \frac{\varphi}{2} + \cos \varphi=0
\]
with $\varphi_1 \in (\frac{\pi}{2}, \pi)$ then there exists $\sigma$ such that $\sigma < \Re(\sqrt{s})^2$ and $s \in \Omega^1_\sigma\cap \Omega^2_\sigma$. Numerically, we find $\varphi_1 \simeq 0.688 \pi$.
\end{proof}

Combining the two results above, we deduce that in the sector $\abs{\arg{s}}<\varphi_0$ the Leaver QNF are guaranteed to be QNF in the sense we have introduced above. Note that we do not assert that all QNF in the sense of Proposition \ref{mainprop} are Leaver QNF, nor even that Leaver QNF necessarily exist. In light of several decades of numerical computation using Leaver's method, it seems reasonable to take this as an empirical fact however.
\appendix
\numberwithin{equation}{section}
\section{Gevrey estimates for $e^{\frac{s}{x}}$}
\renewcommand{\theequation}{A\arabic{equation}}

\begin{Lemma}\label{expLem}
Suppose $\Re(s)<0$. Let $w(x, s)$ be the function given by:
\be
w(x,s) = e^{\frac{s}{x}}
\ee
for $x>0$ and $w(0,s) = 0$. Then for any $\sigma > -\Re(s)$
\be
\sup_{k\in \N}  \sup_{x\in [0,\infty)} \frac{\sigma^k}{(k!)^2} \abs{\frac{d^k}{dx^k}  w(x,s)} = \infty.
\ee
\end{Lemma}
\begin{proof}
Fix $\sigma > -\Re(s)$ and suppose the result fails. Then there exists a uniform constant $C>0$ such that:
\ben{contest}
 \sup_{x\in [0,\infty)} \frac{\sigma^k}{(k!)^2} \abs{\frac{d^k}{dx^k}  w(x,s)} \leq C.
\een
Note that by a standard computation $\frac{d^kw }{dx^k}  (0,s) = 0$. We write $w(x,s) = a(x,s) + i b(x,s)$ for real valued functions $a, b$.

Applying Taylor's theorem at $x=0$ we deduce that for any $k \in \N$
\be
a(x,s) = \frac{x^k}{k!} \frac{d^ka }{dx^k}  (\xi,s) = \frac{x^k}{2 k!} \left(\frac{d^kw }{dx^k}  (\xi,s) +  \frac{d^k \overline{w} }{dx^k}  (\xi,s) \right)
\ee
for some $\xi \in (0, x)$. Making use of \eq{contest} we deduce:
\be
\abs{a(x,s)} \leq C x^k k! \sigma^{-k},
\ee
for all $k$ and $x$ and a similar result holds for $b$. We deduce that:
\be
\abs{w(x,s)} \leq 2C x^k k! \sigma^{-k}.
\ee 
Now, a standard estimate tells us that:
\be
k! \leq (k+1)^{k+1} e^{-k},
\ee 
so that:
\be
\abs{w(x,s)} \leq 2C (k+1)^{k+1} e^{-k} \left(\frac{x}{ \sigma} \right) ^{k},
\ee
for all $x, k$. Setting $x = \frac{\sigma}{k+1}$, we have:
\be
\abs{w\left (\frac{\sigma}{k},s \right)} = e^{ \frac{\Re(s)}{\sigma} k } \leq C (k+1) e^{-k}
\ee
so rearranging we find:
\be
e^{\left (1 + \frac{\Re(s)}{\sigma}\right) k} \leq 2 C (k+1),
\ee
for all $k$, which is a contradiction since by assumption $1 + \frac{\Re(s)}{\sigma} >0$. Thus we must have:
\be
\sup_{k\in \N}  \sup_{x\in [0,\infty)} \frac{\sigma^k}{(k!)^2} \abs{\frac{d^k}{dx^k}  w(x,s)} = \infty.
\ee
\end{proof}

\subsection*{Data availability statement} Data sharing is not applicable to this article as no new data were created or analyzed in this study.

\providecommand{\href}[2]{#2}\begingroup\raggedright\endgroup

\end{document}